\documentclass[11pt,authoryear,letter]{elsarticlehack}

\usepackage{geometry}
\usepackage{amsmath,amssymb,amsfonts, amscd, amsthm,tikz}
\usepackage{siunitx} 
\usepackage{booktabs} 
\usepackage{multirow}
\usepackage[skip=8pt plus1pt, indent=0pt]{parskip}
\usepackage{caption}
\usepackage{subcaption}
\usepackage{enumitem}
\usepackage{placeins}

 \newtheorem{theorem}{Theorem}
 \newtheorem{remark}[theorem]{Remark}
 \newtheorem{definition}[theorem]{Definition}
 \newtheorem{proposition}[theorem]{Proposition}
 \newtheorem{corollary}[theorem]{Corollary}

 \newtheorem{example}{Example}

\usepackage{setspace}

\DeclareCaptionFormat{custom}
{%
    \textbf{#1#2}\textit{\small #3}
}
\newcommand{\abs}[1]{\left\lvert#1\right\rvert}
\usepackage[hypertexnames=false,colorlinks=true,breaklinks=false,bookmarks=true,urlcolor=blue,citecolor=blue,linkcolor=blue,bookmarksopen=false,draft=false]{hyperref}

\allowdisplaybreaks

\journal{--}

\begin{document}

\onehalfspacing

\newif\ifarxiv
\arxivfalse

\ifarxiv

\textbf{\Large 
A Scalable Optimization Approach for Equitable Facility Location: Methodology and Transportation Applications}
\medskip

\else

\begin{frontmatter}

\title{A Scalable Optimization Approach for Equitable Facility Location: Methodology and Transportation Applications}

\author[1]{Drew Horton} 
\ead{drew.horton@ucdenver.edu}
\author[2]{Tom Logan\corref{cor1}} 
\ead{tom.logan@canterbury.ac.nz}
\author[3]{Joshua Murrell}
\ead{5joshmurrell@gmail.com}
\author[3]{Daphne Skipper} 
\ead{daphne.skipper@gmail.com}
\author[1]{Emily Speakman} 
\ead{emily.speakman@ucdenver.edu}

\cortext[cor1]{Corresponding author}

\affiliation[1]{organization={University of Colorado}, 
address={Campus Box 170,
PO Box 173364,},
city={Denver, CO},
zip={80217},
country={USA}}
\affiliation[2]{organization={University of Canterbury}, 
address={Private Bag 4800,},
city={Christchurch}, 
zip={8041},
country={NZ}}
\affiliation[3]{organization={Annapolis, MD}, 
zip={21402},
country={USA}}

\begin{abstract}
Efficient and equitable access to essential services, such as healthcare, food, and education, is an important goal in urban planning, public policy, and transport logistics. However, existing facility location models often do not scale well to large instances, or primarily focus on optimizing average accessibility, neglecting equity concerns, particularly for disadvantaged populations. This paper proposes a novel, scalable framework for equitable facility location, introducing a linearized proxy for the Kolm-Pollak Equally-Distributed Equivalent (EDE) metric to balance efficiency and fairness. Computational experiments demonstrate that our approach scales to extremely large problem instances, while being sensitive enough to account for inequity throughout the distribution, not merely via the maximum value. Moreover, optimal solutions represent significant improvements for the worst-off residents in terms of distance to an open amenity, while also attaining a near-optimal average experience for all users. An extensive real-world case study on supermarket access illustrates the practical applicability of the framework, with additional examples coming from polling applications. As such, the model is extended to handle real-world considerations such as capacity constraints, split demand assignments, and location-specific penalties. By bridging the gap between equity theory and practical optimization, this work offers a robust and versatile tool for researchers and practitioners in urban planning, transportation, and public policy.
 \end{abstract}


\date{}

\end{frontmatter}

\fi


\section{Introduction}\label{intro}
The distribution of facilities within transportation networks shapes how resources and burdens are allocated across communities. From transit stops to emergency services, and from distribution centers to freight hubs, facility location decisions shape distributional justice. Traditional facility location models were designed to address operational problems and, as a consequence, typically provide solutions that are not concerned with the equity of the optimal distribution of distances \citep{Current2002-sb}.  For example, placing facilities to minimize the mean distance tends to leave some community members many miles from service.  On the other hand, prioritizing {\it only} the worst case, i.e., minimizing the maximum distance, ignores the average experience of the community.

Despite these challenges, incorporating equity into facility location models is vital because it can help ensure that resources and services are distributed fairly across a population, reducing disparities, and promoting social justice. This importance has led to a wealth of literature on the subject since the 1970s \citep{OBrien1969-wl,mcallister1976,Savas1978-pk}. Although various approaches have been proposed to incorporate equity in facility location models, including lexicographical optimization of multiple objectives \citep{Ogryczak2009-dx}, these existing methods either do not scale computationally to city-sized instances, cannot evaluate both desirable and undesirable quantities, or lose their normative significance when applied to burdens rather than benefits \citep{Barbati2018-nd}. This has resulted in numerous proposals for incorporating equity considerations \citep{Marsh1994-ax,Eiselt1995-ur,Drezner2009-sx,Ogryczak2009-dx,Lejeune2013-tr,Karsu2015-cb,Barbati2016-rj}, but without consensus on an effective and tractable approach.

The environmental justice literature has recently established clear criteria for metrics that can effectively evaluate distributions of amenities, such as clean air, and burdens, such as distance to an amenity, across a population \citep{LAWC2021}. These (six) properties include the ability to analyze population subgroups, consideration of both relative and absolute differences, and satisfaction of the mirror property -- allowing a consistent evaluation of both goods and their complementary burdens. The Kolm-Pollak Equally-Distributed Equivalent (EDE) has emerged as the preferred metric for environmental justice applications, as it {\it uniquely satisfies all these properties} while providing a normative basis for comparing distributions \citep{MS2020}. However, its non-linear form makes it challenging to optimize, limiting its use in prescriptive facility location models that could help address inequities. This raises a critical question: How can we optimize the Kolm-Pollak EDE in facility location problems while maintaining computational tractability for real-world networks?

We present a solution by providing a computationally tractable approach to optimize the Kolm-Pollak EDE in facility location problems. We prove equivalence between optimizing the non-linear EDE and a linear proxy, enabling the solution of large real-world instances. We extend this formulation to handle split demands and capacity constraints common in transportation applications, and provide guidance for effectively scaling the inequality aversion parameter to achieve desired outcomes. We also develop and analyze an approach for incorporating location-specific penalties, providing theoretical bounds on approximation errors and practical guidelines for implementation. Through computational experiments on instances from the largest U.S. cities, including New York City with over 200 million binary variables, we demonstrate that our approach scales while producing solutions that balance average access with protection of worst-off residents.

While we illustrate our methodology using supermarket access, this is only one example.  Our approach is broadly applicable to transportation facility location problems, including transit stop location, emergency service coverage, and distribution network design. Moreover, the mathematical framework developed for optimizing the Kolm-Pollak EDE could prove valuable for any domain where equity-based optimization of the distribution of resources or burdens is required.



\section{Background and Literature Review}\label{background}
\subsection{Facility Location Models}\label{sec:facility_location_background}
Facility location models optimize where to place services to best serve a population. The decision of where to locate each facility affects the distance that residents must travel, the cost of providing the service, and ultimately who has access. The most basic models aim to either minimize total travel costs or ensure complete coverage of the population. These include set covering models, which minimize the number of facilities needed to serve all residents within a given distance \citep{chvatal1979}; maximal covering models, which maximize the population served within a distance threshold using a fixed number of facilities \citep{Church1974-zz}; p-center models, which minimize the maximum distance any resident must travel \citep{FL1}; and p-median models, which minimize the average distance across the population \citep{FL1,Hakimi1965-ay}. These models have been applied extensively across many domains \citep{Current2002-sb}, but their focus on operational efficiency has important implications for equity.

Our approach builds on the p-median model, which we explain here to establish notation. Consider a region with:
\begin{itemize}
    \item[-] a set of residential areas, $R$ (the ``origins" or ``customers");
    \item[-] a set of potential facility locations, $S$ (points of interest);
    \item[-] a target number of facilities to be opened, $k$.
\end{itemize}

For each residential area $r \in R$, we know:
\begin{itemize}
    \item[-] the population living there,  $p_r$.
    \item[-] the distance to each potential facility location, $d_{r,s}$, for all $s \in S$.
\end{itemize}

Let $T = \sum_{r \in R} p_r$ represent the total population across all residential areas. To track our decisions, we use two sets of binary variables:
\begin{align*}
x_s &:= \begin{cases} 1 & \text{if we open a facility at location } s,\\ 0 & \text{otherwise,} \end{cases} && \forall~ s \in S; \\[1ex]
y_{r,s} &:= \begin{cases} 1 & \text{if we assign residents from area } r \text{ to facility } s,\\ 0 & \text{otherwise,} \end{cases} && \forall~ s \in S, r \in R.
\end{align*}

Using this notation, the standard p-median model can be written as:
\begin{subequations}\label{model:facilitylocation}
\begin{align}
\text{(p-Med)} \quad\quad \text{minimize} ~~~  &\sum_{r\in R}\sum_{s\in S}p_rd_{r,s}y_{r,s}, \label{pmedobj}\\
\text{subject to} ~~~ &\sum_{s \in S} x_s ~=~ k; \label{numfacilities} \\
 &y_{r,s} ~\leq~ x_s, && \forall ~ r \in R, s \in S; \label{openfacilities}\\
 &\sum_{s \in S} y_{r,s} ~=~ 1, &&\forall ~ r \in R; \label{eachorgin1destination}\\
 &\, x_s \in \{0,1\}, && \forall \, s \in S; \label{xbinary} \\
&\, y_{r,s} \in \{0,1\}, && \forall \, r \in R, s \in S. \label{ybinary}
\end{align}
\end{subequations}

The p-median formulation minimizes the average (or equivalently, total) distance traveled by residents to their assigned facility. The constraints ensure that: we open exactly $k$ facilities \eqref{numfacilities}; residents can only be assigned to open facilities \eqref{openfacilities}; and every residential area is assigned to exactly one facility \eqref{eachorgin1destination}. However, by focusing solely on minimizing average distance, this approach can leave some community members unacceptably far from service, highlighting the need for models that explicitly consider equity.

\subsubsection{Incorporating Equity into Facility Location Models}

The goal is to replace the standard p-median objective function, given by
\eqref{pmedobj}, with an alternative function.  The new objective must adequately quantify the equitable access of our population without compromising computational scalability of the facility location model.  Early attempts at this focused on protecting the worst-off residents. The traditional p-center model \citep{Hakimi1965-ay} minimizes the maximum distance that any resident must travel by introducing a single variable $z$ that captures the maximum distance:
\begin{align*}
 \text{(p-Ctr)} \quad\quad \text{minimize} \hspace{0.5cm}& z, \\
   \text{subject to} \hspace{0.5cm} &
   z ~\geq~ d_{r,s}y_{r,s}, \hspace{0.5cm} & \forall ~ r \in R, s \in S; \\
   & \eqref{numfacilities} - \eqref{ybinary}. \nonumber
\end{align*}

While the p-center model ensures no individual is excessively far from a facility, it ignores the experience of the majority of residents. Solutions that appear equitable by protecting the worst-off may provide poor service to the community as a whole. Additionally, the p-center model does not scale computationally to large problems as well as the p-median model.

A natural extension is to minimize a convex combination of the two objective functions. 
Let $\gamma\in [0,1]$, then this appropriately termed ``p-centdian" model \citep{Halpern1976-fz} may be written as: 
\begin{align*}
\text{(p-Ctdn)} \quad\quad \text{minimize} \hspace{0.5cm}& \gamma z+ (1-\gamma)\frac{1}{T}\sum_{r\in R}\sum_{s\in S}p_rd_{r,s}y_{r,s}, \label{pctdnobj} \\
\text{subject to} \hspace{0.5cm} & z ~\geq~ d_{r,s}y_{r,s}, \hspace{0.5cm} & \forall ~ r \in R, s \in S;  
\nonumber \\
   & \eqref{numfacilities} - \eqref{ybinary}. \nonumber
\end{align*}

The value of $\gamma$ allows the user to weight the average distance term and the maximum distance term in the objective function. However, $\gamma$ is not readily interpretable beyond the obvious that larger values lead to the max term taking on greater importance. Moreover, due to the inclusion of this term, the p-centdian does not scale computationally as well as the p-median model.  Along with the p-median and p-center models, we computationally explore the p-centdian model in detail in Section \ref{sec:computational_study}.

The Discrete Ordered Median Problem (DOMP) provides a general framework in which to view (p-Med), (p-Ctr), and (p-Ctdn).  The objective function of the DOMP can be understood as minimizing a weighted sum of the distances from each resident to their closest open facility, i.e., for a resident located in area $r$, this distance is given by $\sum_{s\in S}d_{r,s}y_{r,s}$. The weight given to a particular resident is a penalty that depends on the position of $\sum_{s\in S}d_{r,s}y_{r,s}$ ordered relative to all other residents, and thus, accounts for equity. When all weights are equal, we obtain (p-Med), and at the other end of the spectrum, when all weights {\it not} corresponding to the resident traveling the furthest are set to zero, we obtain (p-Ctr).  

Various linearizations have been developed which allow the DOMP to be formulated as a linear integer optimization problem \citep{10.1007/978-3-642-56656-1_12}.  Moreover, several custom algorithms have been proposed to solve the DOMP; leveraging techniques such as branch and cut \citep{BOLAND20063270, MARIN20091128, MARIN201127,Deleplanque20}, branch and bound \citep{Cherkesly25}, and Benders decomposition \citep{LJUBIC2024858}.
For example, \cite{LJUBIC2024858} present a Benders decomposition method that matches or exceeds the current state-of-the-art on randomly generated instances of DOMP in computational tests. This method solved roughly 25\% of instances of the largest problem size of $|R|=|S|=200$ within a two hour time limit.   Through extensive testing, \cite{Cherkesly25} demonstrate that their specialized branch-and-bound method successfully solves around 25\% of instances of size $|R|=|S|=500$ within a two hour time limit. These include problems equivalent to (p-Med) and (p-Ctr) as well as more general DOMP problems. Here, our goal is to solve the equitable facility location problem on data sets with $|R| = 2067$ and $|S|= 350$ on average, and as large as $|R| = 30,094$ and $|S| = 8,274$.  In order for the methodology to be accessible to a wide range of practitioners, we also wish to use an ``out-of-the-box" integer programming solver, as opposed to a specialized method, therefore we seek an alternative framework to account for equity.

\subsection{Measuring Inequality} \label{sec:quantifying_inequality}

Thus far, we have presented fundamental examples of facility location models focused on (i) efficiency, i.e., p-median; (ii) dispersion, i.e., p-center; and (iii) a parameter-controlled weighting of efficiency and dispersion (the p-centdian).  In addition to these, many other metrics have been proposed to incorporate equity considerations in facility location, although each has important limitations. To enable us to consider the alternatives in context, we first discuss the properties required to appropriately measure equity.

\subsubsection{Properties Required for Equity Assessment}
The chosen location of services and amenities characterizes how a burden, namely travel distance, is distributed across a population.  Thus, these decisions have significant equity implications and substantially impact community well-being, particularly for marginalized populations.  For example, one key contributor to food insecurity is the travel distance required to reach a store offering affordable fresh produce, moreover, this burden is not equitably distributed in the majority of urban areas across the US \citep{Neff09}.  

The question of how best to quantify ``equitable distribution" has long been considered in the environmental justice literature.    More recently, clear criteria have been established for metrics used to evaluate distributions of both amenities and burdens; see \cite{LAWC2021} and the references therein. There are six key properties that an inequality measure should possess if it is used to evaluate the distribution of a singular amenity or burden.  For completeness, we note that a seventh and final property is included, however, this ``multivariate’’ property is not relevant to our discussion and is omitted from the below.  

The first property that characterizes an ideal metric is \emph{symmetry} (or impartiality). The intuitive definition easily translates to a mathematical one; a function is symmetric if its value remains unchanged under any permutation of its arguments.  For example, if person A has an income of \$1 and person B has an income of \$5, an inequality measure should take on the same value as it would if person A had an income of \$5 and person B had an income of \$1.

The second property is \emph{population independence}.  This states that the inequality measure is not influenced by the number of individuals in the community and is therefore appropriate for the comparison of different populations.  

The third property is known as \emph{scale dependence} and ensures that the level of the amenity or burden is taken into account as well as its spread across the community.  An example quickly demonstrates why this property is key.  Without it, an inequality metric for income could declare that distribution I, corresponding to everyone in our population earning \$1, is preferable to distribution II, corresponding to half our population earning \$5 and the other half earning \$10.  However, in environmental justice applications, where lack of an amenity or excess of a burden has a direct impact on well-being, logic dictates we prefer distribution II, where every community member is comparatively better off.  

The fourth property is called \emph{the principle of transfers} and guarantees that if we redistribute some fixed quantity from a relatively advantaged member of our community to a disadvantaged member, then our measure should improve (assuming no other changes occurred).  

The fifth condition states that the inequality measure should satisfy the \emph{mirror property} which ensures the measure can appropriately deal with amenities \emph{and} burdens.  Failure to meet this property can lead to unexpected consequences. We would like our ordering of distributions by preference to remain the same irrespective of whether we consider an amenity directly or its complementary burden.  For example, the same distribution should be preferred regardless of whether we consider the `burden’ of carbon dioxide in the air or the `amenity’ of clear air measured by a lack of carbon dioxide.  It is the mirror property that ensures this consistency.   

Finally, the sixth property is \emph{separability} and enables meaningful comparison between subgroups of our population.  Importantly, it allows us to assess disparities between demographic groups.

\subsubsection{Review of Existing Measures (Suitability and Scalability)}

As we have seen, the maximum value metric, used in p-center models, optimizes the experience of the worst-off community member. In doing this, both the maximum value and range metrics suffer from considering only extreme values while ignoring the overall distribution. Moreover, neither satisfies the principle of transfers. Several other measures attempt to capture dispersion by examining deviation from the mean, including mean absolute deviation, variance, and the variance of logarithms. However, these also fail to satisfy the principle of transfers \citep{Allison1978-lk,Erkut1993-kw}.

More sophisticated inequality measures have also been explored, including Schutz's Index, the coefficient of variation, the Gini coefficient, Theil's entropy coefficient, and Atkinson's coefficient. While some of these metrics satisfy the transfers principle, they have other critical drawbacks. The Gini coefficient, despite its widespread use in facility location models \citep{Mandell1991-by, Mulligan1991-zf, Drezner2009-sx, Lejeune2013-tr, Alem2022-rp}, measures only dispersion without considering the overall level of the distribution. Thus, it may prefer a solution where all residents are equally far from a facility over one with shorter but varied distances -- clearly problematic for accessibility. The coefficient of variation and Theil's coefficient suffer from the same limitation.  This problem, i.e., violating scale dependence, is often overcome by augmenting the preferred measure of dispersion with an efficiency metric, a key example of this being the p-centdian objective function.  However, it is often not clear how to ``best" balance dispersion and efficiency metrics. In addition, the resulting measure is not guaranteed to satisfy all six of the properties necessary to assess equity.

Alternative approaches proposed to incorporate equity in facility location models include lexicographical optimization of multiple objectives \citep{Ogryczak2009-dx} and general p-norm optimization approaches \citep{Karsu2015-cb,Olivier2022,Gupta23,Xinying_Chen2023}. However, these methods either do not scale computationally, cannot evaluate both desirable and undesirable quantities, or lose their normative significance when applied to burdens rather than benefits \citep{Barbati2018-nd}.

Equally-distributed equivalents (EDEs) offer a promising direction as they consider both the center and spread of a distribution in a single measure. The Atkinson EDE \citep{EDEAt} has been applied to facility location, but makes the implicit assumption that higher values represent better welfare. Thus, it is inappropriate for measuring burdens such as travel distance because it does not satisfy the mirror property \citep{Cox2012}.  Furthermore, computational studies have shown that optimizing the Atkinson EDE becomes intractable for realistically-sized problems \citep{Barbati2018-nd}.

Indeed, computational tractability emerges as a major barrier for many of the proposed equity metrics discussed. Recent comprehensive studies found that several measures 
are too complex to incorporate into optimization models that can handle real-world problem sizes \citep{Barbati2018-nd, Xinying_Chen2023}.  

As a concrete example, we present Table \ref{tab:comparisonB2018} taken directly from \cite{Barbati2018-nd}, and adapted only for the purposes of consistent notation.  Observe that in the largest instances with $|R|=|S|=60$ and $k=10$, only two of the ten measures considered could be solved to optimality within the 2 hour time limit for all instances. Recall that the problems we are interested in are orders of magnitude larger; hence the computational results suggest that we will be unable to optimize the majority of the measures, either directly or when combined with an efficiency measure.

\begin{table}[ht]
\centering
\begin{tabular}{l|cc|cc|cc}
\hline
Measure & \multicolumn{2}{c|}{$|R| = |S| = 20$} & \multicolumn{2}{c|}{$|R| = |S| = 40$} & \multicolumn{2}{c}{$|R| = |S| = 60$} \\
\cline{2-7}
& $k = 2$ & $k = 10$ & $k = 2$ & $k = 10$ & $k = 2$ & $k = 10$ \\
\hline
Center (Maximum) & $<1$ & $<1$ & 16 & 4 & 43 & 20 \\
Range & 1 & 4 & 17 & 21 & 464 & 269 \\
Mean Absolute Deviation & $<1$ & $<1$ & 24 & 264 & 166 & N/A \\
Variance & 3 & $<1$ & 205 & 810 & 1206 & N/A \\
Maximum Absolute Deviation & 3 & $<1$ & 7 & 6 & 143 & N/A \\
Absolute Difference & 6 & 2 & 46 & 2656 & 1123 & N/A \\
Sum Maximum Difference Absolute & 13 & 3 & 56 & 767 & 1004 & N/A \\
Schutz's Index & 3 & 3 & 81 & N/A & 1033 & N/A \\
Coefficient of Variation & 4 & 22 & 130 & N/A & 1588 & N/A \\
Gini Coefficient & 5 & 7 & 182 & N/A & 4022 & N/A \\
\hline
\end{tabular} 
\smallskip

\caption{\citep{Barbati2018-nd}. Average solve time (in seconds) over 100 randomly generated test problems with $|R|$ demand points and $|S|$ potential facility locations, selecting $k$ locations. ``N/A'' denotes that an optimal solution could not be found for at least one of the instances within the time limit of 2 hours.  Computations were performed using the solver Cplex v12.00 on a Pentium IV with 2.40 GigaHertz and 4.00 GigaBytes of RAM.}
\label{tab:comparisonB2018}
\end{table}

The lack of a metric that demonstrably satisfies all six properties necessary to assess equity, in addition to computational challenges, had left the field of environmental justice without consensus on an appropriate and practical approach to measure equity in facility location problems.
This fact motivated the development of the Kolm-Pollak EDE, which uniquely satisfies all key properties required for equity assessment in environmental justice applications. However, like many previous measures, its nonlinear form has limited its use to descriptive analysis rather than prescriptive optimization. This research bridges this gap. By developing a computationally tractable approach to optimizing the Kolm-Pollak EDE, we enable its application to large-scale facility location problems, allowing transportation planners to not just measure inequity but actively design for equity.

\subsubsection{The Kolm-Pollak EDE} \label{kolmpollakede}

The Kolm-Pollak Equally-Distributed Equivalent (EDE) was introduced by \cite{MS2020} based on Kolm–Pollak preferences \citep{Kolm1976-ab}, and is uniquely suited for evaluating distributions of both amenities and burdens in transportation systems. Unlike previous measures, it satisfies {\it all six} properties required for equity assessment. 

When applied to travel distances, the Kolm-Pollak EDE represents the uniform distance that, if experienced by all residents, would generate the same level of social welfare as the actual unequal distribution. This can be interpreted as the mean distance plus a penalty for inequality, where the penalty increases with greater disparity in access. Mathematically, for a distribution of distances $z_1, z_2, \dots, z_N$, the Kolm-Pollak EDE is defined as:
\begin{align}
\mathcal{K}(\mathbf{z}) := -\frac{1}{\kappa}\ln\left[\frac{1}{N}\sum_{i=1}^{N}e^{-\kappa z_i}\right].
\end{align}

The parameter $\kappa := \alpha\epsilon$ incorporates society's aversion to inequality through $\epsilon$, where the sign of $\epsilon$ depends on whether larger or smaller values of the quantity are preferred. For desirable quantities like income or accessibility, where more is better, $\epsilon > 0$ and the Kolm-Pollak EDE will be less than or equal to the mean. Conversely, for undesirable quantities like travel distance or exposure to pollution, where less is better, $\epsilon < 0$ and the EDE will be greater than or equal to the mean. In either case, the magnitude $|\epsilon|$ (typically between 0.5 and 2), represents the strength of aversion to inequality. Here $\alpha$ is a normalization factor, translating $\epsilon$ into $\kappa$:
\begin{align}
\alpha := \frac{\sum_{i=1}^{N} z_i}{\sum_{i=1}^{N}z_i^2}.
\end{align}

\autoref{fig:ede} illustrates this concept for a distribution of travel distances to facilities or resources. The histogram shows how distances are distributed across the population, with the mean distance indicated by a vertical line. Since travel distance is undesirable, we use $\epsilon < 0$ resulting in an EDE that exceeds the mean. The gap between the EDE and mean represents the inequality penalty -- larger values of $|\epsilon|$ increase this penalty, reflecting greater concern for residents experiencing particularly long travel times. The EDE provides transportation planners with a single metric that captures both average system performance and distributional fairness.

\begin{figure}[h!]
    \centering
    \includegraphics[width=0.7\textwidth]{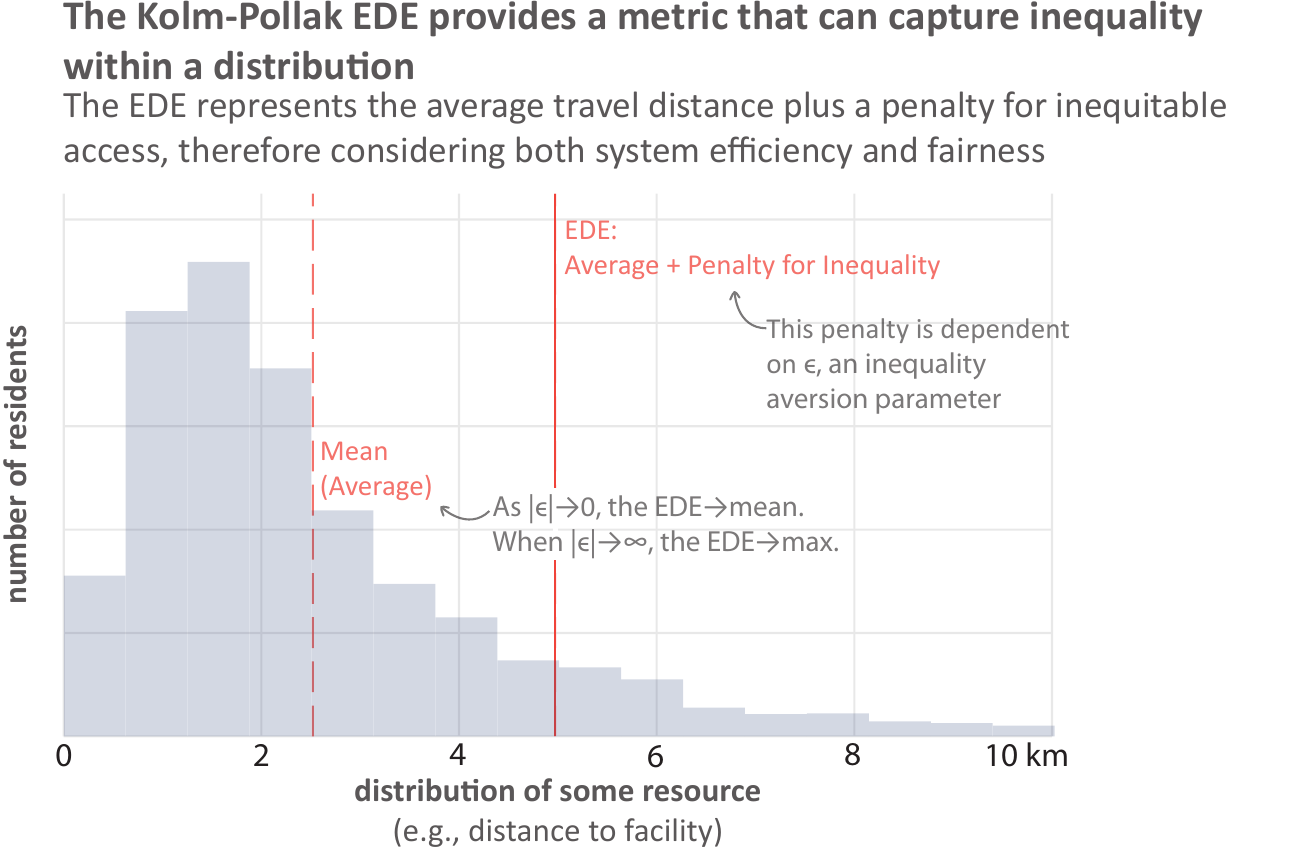}
    \caption{This figure shows a hypothetical distribution of residents' distance to a resource, with the average (mean) and Kolm-Pollak Equally-Distributed Equivalent (EDE) indicated by vertical lines. The inequality penalty—shown by the gap between the mean and EDE—increases with the value of the inequality aversion parameter and reflects the presence of residents who are significantly worse off than average.}
    \label{fig:ede}
\end{figure}

In addition to satisfying the required properties, a key advantage of the Kolm-Pollak EDE is how it bridges between conventional facility location objectives. As $\epsilon$ approaches 0, the EDE approaches the mean distance ($p$-median objective), while as $\epsilon$ approaches negative infinity, it approaches the maximum distance ($p$-center objective). This provides a principled way to balance average accessibility against the experience of the worst-off residents.  Moreover $\epsilon$, the so-called ``inequality aversion" parameter, is consistent with other EDE functions from the literature, for example the Atkinson EDE, and is well-understood by the environmental justice community.  

\begin{example} 
Table \ref{tab:psuedo-data} demonstrates properties of the Kolm-Pollak EDE using four example distance distributions sharing the same mean but having varying inequality. As disparity increases, the EDE grows larger than the mean, with the gap widening for larger values of $|\epsilon|$. This captures how inequitable access reduces overall system performance from an environmental justice perspective.

\begin{table}[ht]
\setlength{\tabcolsep}{7pt}
\centering

\begin{tabular}{c|cccccccccc}
\toprule 
Distribution  & $z_1$ & $z_2$ & $z_3$ & $z_4$ & Mean & St. Dev. & Max & \multicolumn{3}{c}{Kolm-Pollak} \\
&&&&&&&& $\epsilon = -1$ & $\epsilon = -2$ & $\epsilon = -50$ \\[0.1em]
 \midrule
1 & 100 & 100 & 100 & 100 & 100 & 0 & 100 & 100 & 100 & 100\\
2 & 50 & 75 & 125 & 150 & 100 & 46.5 & 150 & 106.7 & 112.7 & 146.8 \\
3 & 0 & 0 & 200 & 200 & 100 & 115.5 & 200 & 124.0 & 143.4 & 197.2\\
4 & 0 & 0 & 0 & 400 & 100 & 200 & 400 & 142.9 & 190.9 &389.0\\
    \bottomrule 
\end{tabular}

\smallskip

\caption{The mean distance is 100 for all four distributions, while the Kolm-Pollak scores are larger for less equitable distributions, increasingly so as $\abs{\epsilon}$ increases.  For large values, $\abs{\epsilon}=50$, we observe that the Kolm-Pollak scores are approaching the maximum of the distribution. }
\label{tab:psuedo-data}
\end{table}

\end{example}

\section{A Framework for Equitable Facility Location} \label{EFLP}
\subsection{Base Model Formulation}
Transportation planners and geographers need to evaluate how equitably a network serves its population. The Kolm-Pollak EDE addresses this by quantifying accessibility while penalizing disparities in service. To apply this metric to facility networks, we measure the distance each resident must travel to reach their nearest facility, weighted by population density since residential data is typically aggregated at the census block level (the USA's smallest census reporting area). Let $p_r$ represent the population of residential area $r \in R$, and let $z_r$ represent the distance that residents of $r$ must travel to reach a facility. The Kolm-Pollak EDE for this distribution is,
\begin{align} \label{popweightedkp}
\mathcal{K}(\mathbf{z}) = -\frac{1}{\kappa}\ln\left[\frac{1}{T}\sum_{r \in R}p_re^{-\kappa z_r}\right],
\end{align}
where $\mathbf{z} \in \mathbb{R}^{|R|}$ is the vector of travel distances, $T:=\sum_{r \in R}p_r$ is the total population over all residential areas, $\kappa := \alpha\epsilon$, and
\begin{equation} \label{eq:alpha}
    \alpha ~:=~ \frac{\sum_{r \in R} p_rz_r}{\sum_{r \in R} p_r(z_r)^2}.
\end{equation}
We model the distance traveled by residents of $r \in R$ as $z_r := \sum_{s\in S}y_{r,s}d_{r,s}$ (where $y_{r,s}$ and $d_{r,s}$ are as defined in Section \ref{background}), so we have 
\begin{align}
\label{eq:KP}
\mathcal{K}(\mathbf{y}) = -\frac{1}{\kappa}\ln\left[\frac{1}{T}\sum_{r\in R}p_re^{-\kappa \sum_{s\in S}y_{r,s}d_{r,s}}\right],
\end{align}
with $\mathbf{y}\in \mathbb{R}^{|R|\times |S|}$.
We can simplify this expression for use in an optimization model, which is a necessary step given the size of the instances we wish to solve.

The first thing to note is that $\alpha$, and therefore $\kappa = \alpha\epsilon$, depends on our optimization variables. However, we can greatly simplify the optimization by choosing to treat $\alpha$ as a constant. Fortunately, we can do this and still be minimizing an actual Kolm-Pollak EDE. 

Recall that $\alpha$ scales $\epsilon$, the chosen inequality aversion parameter, for use with the precise distribution we wish to evaluate.   This means that the ``true'' value of $\alpha$ corresponds to the distribution given by an optimal solution to the optimization model, which we do not have a priori.  However, we can estimate this distribution, and therefore $\alpha$, using the problem data.  

Crucially, we can quantify exactly how using an estimate for $\alpha$ impacts our solution. If $\alpha$ were expressed as a function of the decision variables in the objective function, then the aversion to inequality used to obtain the optimal solution would exactly match the $\epsilon$ selected by the user. The consequence of fixing $\alpha$ at the outset is that the user does not have complete control over the chosen level of aversion to inequality.  For example, the user may have intended to choose an inequality aversion level of $\epsilon = -1$, however, because this is only scaled approximately, in reality, the solution obtained by the model optimizes a different level of inequality aversion.  Once we have obtained this optimal solution, we can calculate the true value of $\epsilon$ and compare it to the level of inequality aversion that was intended.  Given that the choice of inequality aversion is somewhat subjective, it may be that we are happy with the approximation.  If not, we can re-solve the model with a more accurate $\alpha$-approximation. Extensive computational tests yielded a near-perfect level of inequality aversion on the second run of the model in every instance.  See Section \ref{sec:scaling_parameter} for a detailed discussion. 

Once we assume that $\alpha$ (and therefore $\kappa$), is constant, we are able to claim that optimizing $\mathcal{K}(\mathbf{y})$ is equivalent to optimizing the function, \begin{equation*}
    \breve{\mathcal{K}}(\mathbf{y}):=\sum_{r \in R}p_re^{-\kappa \sum_{s \in S}y_{r,s}d_{r,s}}.
\end{equation*} 
In this context, equivalent means that the optimal solution set is the same in both cases and that it is easy to recover the optimal value of $\mathcal{K}$ from the optimal value of $\breve{\mathcal{K}}$.

First recall that in our setting, $\epsilon < 0$, and so we can remove the positive constant $-\frac{1}{\kappa}$ for optimizing $\mathcal{K}$. Then observe that because the natural logarithm function is monotonically increasing, we can equivalently minimize the argument of the function.  Finally, dropping the positive constant $\frac{1}{T}$ gives us $\breve{\mathcal{K}}$. We formally define the Kolm-Pollak facility location model (KP) as follows, using the same notation as defined for (p-Med) above:
\begin{align} \label{kpobj}
  \text{(KP)} \quad\quad \text{minimize} \hspace{0.5cm}& \breve{\mathcal{K}}(\mathbf{y}) :=  \sum_{r \in R}p_re^{-\kappa \sum_{s \in S}y_{r,s}d_{r,s}},\\
    \text{subject to} \hspace{0.5cm} &
    \eqref{numfacilities} - \eqref{ybinary}. \nonumber
\end{align}

\begin{proposition} \label{prop:equivalent} Suppose $y_{r,s} \in \{0,1\}$ for all $s \in S$, $r\in R$, and  $\sum_{s\in S} y_{r,s} 
= 1$ for all $r \in R$. Then 
\begin{equation*}
    \sum_{r \in R}p_re^{-\kappa \sum_{s \in S}y_{r,s}d_{r,s}}\\
    ~=~\sum_{r \in R}\sum_{s \in S}p_ry_{r,s}e^{-\kappa d_{r,s}}.
\end{equation*}
\end{proposition}
For the proof, see \ref{sec:app}.

Replacing $\breve{\mathcal{K}}$ with the linear version, which we denote $\overline{\mathcal{K}}$, we obtain the Kolm-Pollak linear proxy model: 
\begin{align} 
  \text{(KPL)} \quad\quad \text{minimize} \hspace{0.5cm}& \overline{\mathcal{K}}(\mathbf{y}) :=  \sum_{r \in R}\sum_{s \in S}p_r y_{r,s}e^{-\kappa d_{r,s}}, \label{kplpobj} \\
    \text{subject to} \hspace{0.5cm} &
    \eqref{numfacilities} - \eqref{ybinary}. \nonumber
\end{align}
\begin{corollary}\label{cor:linearkp}
(KPL) has the same set of optimal solutions as (KP). 
\end{corollary}
An optimal Kolm-Pollak EDE score (using the approximate value for $\alpha$), ${\mathcal{K}}^*$, can be recovered from an optimal objective value of (KPL), $\overline{\mathcal{K}}^*$, as follows:
\begin{equation}
\mathcal{K}^* ~=~ -\frac{1}{\kappa}\ln\left(\frac{1}{T}\overline{\mathcal{K}}^*\right).
\end{equation}

\subsection{Model Extensions for Practice} 
\label{sec:modelextensions}
In this section, we discuss how a few common facility location modeling scenarios can be applied within the (KPL) model. 

\subsubsection{Facility Cost} \label{sec:facilitycosts}
Suppose we wish to incorporate a ``fixed-charge" of $c_s$ for opening facility $s \in S$. The simplest way to incorporate facility fixed-charges into the (KPL) model is to add a budget constraint of the form,
$\sum_{s \in S} c_s x_s \leq b$,
 where $b$ represents the overall facility budget. This constraint might replace constraint \eqref{numfacilities}, depending on the modeling scenario. 

If the $c_s$ and $d_{r,s}$ parameters represent values with the same units (dollars or meters, for example), then we might wish to incorporate a fixed-charge term into the (KPL) objective function directly. However, because the $d_{r,s}$ values do not appear in their standard units in the (KPL) objective function (due to dropping the log function), we must take care to add the fixed-charges, $c_s$, so that they have the appropriate weight relative to the $d_{r,s}$ values. We tackle this modeling scenario in detail in Section \ref{sec:penalty}. If $c_s$ and $d_{r,s}$ represent different values (say cost and distance, respectively), then a multi-objective approach would more appropriate for capturing the fixed-charge expression as an additional objective; see \cite{Cho2017-multiobectivesurvey} for a survey of multi-objective optimization techniques.

\subsubsection{Facility Capacity}
\label{splitdemands}

Facility capacities are a natural consideration in many facility location applications. If $C_s$ represents the number of customers that can be served at facility $s$, the constraints, 
\begin{equation} \label{cons:capacity}
\sum_{r\in R} p_r y_{r,s} ~\leq~ C_s x_s, ~~\forall ~ s \in S,
\end{equation}
ensure that the capacity of every opened facility is respected by feasible solutions. 

\subsubsection{Split Demand Assignment}
In some applications, it makes sense to allow the total demand of customer $r \in R$ to be served by multiple open facilities. Allowing for split demands is a natural consideration when service facilities have capacities, and it might not be possible for each customer to be assigned to the closest open facility. We allow for split demands by relaxing the integrality of the $y_{r,s}$ variables; i.e., by replacing constraints \eqref{ybinary} with $y_{r,s} \in [0,1]$. 

Interestingly, when the $y_{r,s}$ variables are relaxed to the continuous interval $[0,1]$, Proposition \ref{prop:equivalent} no longer holds, but (KPL) may actually be preferred over (KP) as the more accurate model. Recall that for a given $r \in R$, constraints \eqref{eachorgin1destination} ensure $\sum_{s \in S} y_{r,s} = 1$, and first consider the original, nonlinear Kolm-Pollak expression,
\begin{equation}
-\frac{1}{\kappa}\ln\left[\frac{1}{T}\sum_{r \in R}p_re^{-\kappa \sum_{s \in S}y_{r,s}d_{r,s}}\right]. \tag{\ref{eq:KP}}
\end{equation}

For a fixed $r \in R$, the expression $\sum_{s \in S}y_{r,s}d_{r,s}$ represents a weighted average of the distances from $r$ to each of the facilities assigned to $r$. The distribution of distances measured by expression \eqref{eq:KP} assumes this weighted average distance for each of the $p_r$ people residing in census block $r$, and as such, inequality between residents in the {\it same} census block is not accounted for in the Kolm Pollak score. Poor distances are obscured because the weighted average is taken before the Kolm-Pollak expression is applied. However, this is not the case in the expression minimized by (KPL).

Substituting the linear proxy into the Kolm-Pollak EDE, we obtain the expression that is minimized by (KPL):
\begin{equation}
\label{eq:KPalt}
-\frac{1}{\kappa}\ln\left[\frac{1}{T}\sum_{r \in R}\sum_{s \in S} p_r y_{r,s}e^{-\kappa d_{r,s}}\right].
\end{equation}
In this version of the Kolm-Pollak expression, the weighted average defined by $y_{r,s}$ for a given $r$ is applied to the collection of Kolm-Pollak terms of the form $p_r e^{-\kappa d_{r,s}}$, as opposed to the distances directly. In other words, (KPL) models the scenario as if $p_r y_{r,s}$ people travel to facility $s$ at a distance of $d_{r,s}$, and therefore {\it does} account for any inequality between residents in the same census block. Note that because $p_r y_{r,s}$ may be fractional, the expression minimized by (KPL) is not technically a Kolm-Pollak score. However, the expression is a more accurate model of the Kolm-Pollak EDE any time we wish to take into account the inequality {\it within} a given residential area because distances are appropriately penalized. See Example \ref{ex:KPsplitdemands} for a concrete illustration of these concepts.

\begin{example}
\label{ex:KPsplitdemands}
Consider a single residential area with a total population of 100.  Suppose half of the population are assigned to a facility located at a distance of 1 unit away and half are assigned to a facility located 100 units away as illustrated in Figure \ref{fig:split_demands}.

\begin{figure}[ht]
    \centering
   
    \includegraphics[width=0.8\textwidth]{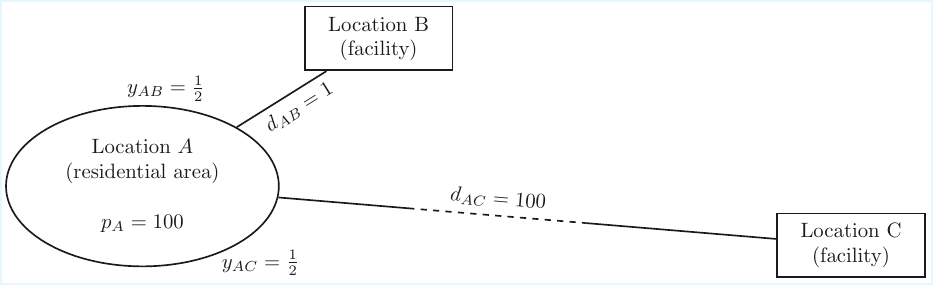}
    \caption{Illustration of the distribution in Example \ref{ex:KPsplitdemands}.}
    \label{fig:split_demands}
\end{figure}

For this distribution, we compute $$\alpha =\frac{(50\times1)+(50\times 100)}{(50\times1^2)+(50\times 100^2)} \approx 0.0101, $$ and setting $\epsilon = -1$, obtain $\kappa = -0.0101.$ 

Using these values, we substitute into expression \eqref{eq:KP}, and calculate the Kolm-Pollack score of the distribution to be 
$$-\frac{1}{\kappa}\ln\left[\frac{1}{100}\times 100 \times e^{-\kappa\Big (\big ( \frac{1}{2}\times 1 \big ) + \big(\frac{1}{2} \times 100\big )\Big )}\right] = 50.5. $$
Observe that the value obtained is precisely the mean of the distribution.  Moreover, this occurs because the model assumes that each person residing at location A travels the same distance to reach a facility, at least on average.  As there is only one residential area in this simple example much of the expression cancels, and the Kolm Pollak score reduces to the mean. This is expected because with only one residential area under consideration, there can be no inequality between different residential areas, and inequality within the same residential area is not accounted for when using expression \eqref{eq:KP}.

In contrast, substituting into equation \eqref{eq:KPalt}, we obtain a value strictly greater than the mean, accounting for the inequality present within our single residential area:
$$-\frac{1}{\kappa}\ln\left[\frac{1}{100}\times 100\times \Bigg (  \left(\frac{1}{2}\times e^{-\kappa\times 1} \right)+ \left (\frac{1}{2}\times e^{-\kappa\times 100} \right )\Bigg )\right] \approx 62.39. $$
Using expression \eqref{eq:KPalt}, we are able to model the inequality in the proposed scenario by making the implicit assumption that 50 people travel a distance of 1 to reach a facility, and 50 people travel a distance of 100 to reach a facility. 
\end{example}

Note that (KPL) with the $y_{r,s}$ variables relaxed has the same structure as (KPL) except with many binary variables replaced by continuous variables, so we would expect its computational performance to be at least as good as the computational performance of (KPL).

\subsection{Incorporating Location Preferences Through Penalties}
\label{sec:penalty}
Note that for clarity of presentation, proofs of all theorems and propositions given in this section can be found in \ref{sec:app}.

It is rarely the case that all potential locations are equally suitable. For example, in the equitable polling locations application, early voting sites are open on many days leading up to an election and this can be disruptive to normal operations at the selected sites. When this model was used by advocacy groups in preparation for the 2024 United States presidential election, fire stations and churches were supplied as potential early voting sites in one Georgia county, but were only to be selected if their inclusion would be particularly impactful with respect to voter access \citep{Skipper2025}. Similarly, if an urban planning board for a particular city were selecting supermarket locations to improve food access, some potential locations would be less suitable than others due to land costs or other practical considerations.

In this section, we present a method for applying a facility-based fixed-charge penalty, $c_s$, in the same units as the $d_{s,r}$, so that penalized location $s \in S$ will be selected only if it improves the optimal Kolm-Pollak score by at least $c_s$ units. This model has two avenues by which error in the size of the distance penalty may be introduced. Careful analysis of both types of error leads to a modeling strategy aimed at keeping the penalty error low.

\subsubsection{Mathematical Formulation}

Suppose $U \subseteq S$ is a set of less-desirable potential locations. For each $s \in U$, let $c_s \geq 0$ represent the number of units of improvement in the Kolm-Pollak score that would need to be realized (versus an optimal score not using site $s$) to select location $s$. We want to incorporate these preferences in the (KPL) model by penalizing the optimal Kolm-Pollak score by $\sigma:=\sum_{s \in U} c_s x_s$ units. However, we cannot do this directly because the objective function, given by \eqref{kplpobj}, does not minimize the Kolm-Pollak score directly. The following result provides the value that the (KPL) linear objective function, $\overline{\mathcal{K}}$, must be penalized by for the associated Kolm-Pollak score, $\mathcal{K}$, to be penalized by $\sigma$ meters.
\begin{theorem} \label{thm:penalty}
Let $\overline{\mathcal{K}}$ represent the unpenalized objective value of (KPL). Let $\mathcal{K}$ represent the associated unpenalized Kolm-Pollak score: $\mathcal{K} = -\frac{1}{\kappa}\ln\left(\frac{1}{T}\overline{\mathcal{K}}\right)$. Then adding a penalty of,
\begin{equation}\label{eq:penalty}
\rho ~:=~ Te^{-\kappa \mathcal{K}}(e^{-\kappa \sigma} -1),
\end{equation}
to $\overline{\mathcal{K}}$ is equivalent to adding a penalty of $\sigma \geq 0$ units to $\mathcal{K}$.
\end{theorem}

Now we have the following model for penalizing less-desirable potential service locations.  The new variable $v$ captures the nonlinear portion of the penalty via the pressure of the minimizing objective:
\begin{align}
\text{(KPL$^p$)}\quad\quad \text{minimize} \hspace{0.5cm}& \overline{\mathcal{K}}^p(\mathbf{x},\mathbf{y}) 
  ~:=~   \overline{\mathcal{K}}(\mathbf{y})
 + Te^{-\kappa \hat{\mathcal{K}}}(v-1), \label{obj:penalized} \\
\text{subject to} ~~~&v ~\geq~ e^{q}; \label{cons:nl_penalty}\\
& q ~=~ -\kappa \sum_{s \in U} c_s x_s; \label{cons:q} \\
& \eqref{numfacilities} - \eqref{ybinary}. \nonumber
\end{align}
In the objective function \eqref{obj:penalized}, $\hat{\mathcal{K}}$ is a parameter approximating $\mathcal{K}^*$, the unpenalized Kolm-Pollak score associated with an optimal solution to (KPL$^p$):  $\mathcal{K}^* := -\frac{1}{\kappa}\ln\left(\frac{1}{T}\overline{\mathcal{K}}(\mathbf{y}^*)\right)$, where $(\mathbf{x}^*, \mathbf{y}^*)$ is optimal to (KPL$^p$). Moreover, to use a linear optimization solver, we must linearize the one-dimensional exponential function in constraint $\eqref{cons:nl_penalty}$. Let $\sigma^{max}$ represent the largest possible value that $\sigma = \sum_{s\in U} c_s x_s$ can take for a feasible solution, $(\mathbf{x},\mathbf{y})$, to (KPL$^p$). We can accomplish the linearization by replacing constraint \eqref{cons:nl_penalty} with a set of lower-bounding tangent lines constructed at points $(q, e^q)$, for $q \in \boldsymbol{\beta} = \{\beta_0, \beta_1, \dots, \beta_n\}$, where $0=\beta_0 < \beta_1 < \dots < \beta_{n-1} < -\kappa\sigma^{max} \leq \beta_n$, as follows:
\begin{align}
  \text{(KPL$^t$)} \quad\quad \text{minimize} \hspace{0.5cm}& \overline{\mathcal{K}}^p(\mathbf{x},\mathbf{y}) 
  ~:=~   \overline{\mathcal{K}}(\mathbf{y})
 + Te^{-\kappa \hat{\mathcal{K}}}(v-1), \tag{\ref{obj:penalized}} \\
\text{subject to} ~~~&v ~\geq~ e^{\beta_i} + e^{\beta_i}(q-\beta_i), ~~~\forall~i \in \{0,1,\dots,n\}; \label{cons:tangent_lines}\\
& q ~=~ -\kappa \sum_{s \in U} c_s x_s; \tag{\ref{cons:q}} \\
& \eqref{numfacilities} - \eqref{ybinary}. \nonumber
\end{align}

The accuracy of the penalty term depends both on the quality of the approximation $\hat{\mathcal{K}}$ in the objective function \eqref{obj:penalized} and on the linearization points, $\boldsymbol{\beta}$, chosen for constraint \eqref{cons:tangent_lines}. In the rest of this section, we analyze the error to the penalty term introduced by each of these approximations, which leads to guidance on choosing $\hat{\mathcal{K}}$ and $\boldsymbol{\beta}$ to keep this error low. 

\subsubsection{Error Analysis and Implementation Guidelines}
First we consider the {\bf penalty error introduced by parameter approximation}.  Let $\sigma^* := \sum_{s \in U} c_s x_s^* \geq 0$ represent the penalty associated with stores selected by optimal solution $(\mathbf{x}^*, \mathbf{y}^*)$ to (KPL$^p$). Let $\hat{\sigma}$ represent the actual penalty (approximating $\sigma^*$) that is applied to an optimal Kolm-Pollak score by (KPL$^p$). Let $\Delta := \hat{\mathcal{K}} - \mathcal{K}^*$ represent the signed error in approximating $\mathcal{K}^*$ by $\hat{\mathcal{K}}$. 

\begin{theorem} \label{thm:penalty_approx_error}
The error in the penalty applied to an optimal Kolm-Pollak score by (KPL$^p$) is,
\begin{equation}
\hat{E}(\Delta,\sigma^*) 
~:=~ \hat{\sigma} - \sigma^* 
~=~ -\frac{1}{\kappa}\ln\left(e^{\kappa\sigma^*}+e^{-\kappa\Delta}(1-e^{\kappa\sigma^*})\right).
\end{equation}
\end{theorem}

The following result provides some useful properties of the penalty approximation error, $\hat{E}(\Delta, \sigma^*)$.

\begin{theorem} \label{thm:penalty_approx_error_bound}
Suppose $\sigma^* > 0$ is fixed:
\begin{enumerate}[label=(\arabic*)]
\item $\hat{E}(0,\sigma^*) = 0$ and $|\hat{E}(\Delta,\sigma^*)|$ is increasing in $|\Delta|$.
\item If $\hat{\mathcal{K}} > \mathcal{K}^*$ then $\hat{\sigma} > \sigma^*$ (locations are over-penalized) and 
\[
0
< |\Delta|(1-e^{\kappa \sigma^*}) 
< |\hat{E}(\Delta,\sigma^*)|
< |\Delta|.
\]
\item If $\hat{\mathcal{K}} < \mathcal{K}^*$ then $\hat{\sigma} < \sigma^*$ (locations are under-penalized) and 
\[
0
< |\hat{E}(\Delta,\sigma^*)|
< \text{min}\left\{|\Delta|(1-e^{\kappa \sigma^*}),\sigma^*\right\}.
\]
\end{enumerate}
Suppose $\Delta \in \mathbb{R}$ is fixed:
\begin{enumerate}[resume*]
\item $\hat{E}(\Delta, 0) = 0$ and $|\hat{E}(\Delta, \sigma^*)|$ is increasing in $\sigma^*$.
\end{enumerate}
\end{theorem}

\begin{remark}
Some consequences of Theorem \ref{thm:penalty_approx_error_bound} are as follows:
\begin{itemize}
\item The bounds in the $\Delta > 0$ case imply that $\hat{E}(\Delta,\sigma^*)$ is unbounded as $\Delta$ approaches infinity. In fact, it can be shown that $\hat{E}(\Delta,\sigma^*)$ approaches $\Delta - \sigma^* - \frac{1}{\kappa}(e^{-\kappa \sigma^*} -1)$ asymptotically as $\Delta$ approaches infinity. 
\item The upper bound on the size of the error in the $\Delta < 0$ case implies that there will never be a negative penalty (i.e., a ``bonus''), for selecting an undesireable location.
\end{itemize}
\end{remark}

There are two natural choices for $\hat{\mathcal{K}}$: $\mathcal{K}^{all}$ and $\mathcal{K}^{rem}$, as defined below. 

\begin{definition}
Let (KPL)$^{all}$ denote the unpenalized model, (KPL), with all potential locations from (KPL$^p$) included. Let (KPL)$^{rem}$ denote (KPL) with all penalized potential locations removed from consideration, and assume that the there are enough remaining locations (at least $k$) so that the model is feasible. Suppose $(\mathbf{x}^*, \mathbf{y}^*)$ is optimal to (KPL)$^{all}$ and define $\mathcal{K}^{all} := -\frac{1}{\kappa}\ln\left(\overline{\mathcal{K}}(\mathbf{y}^*)\right)$. Define $\mathcal{K}^{rem}$ similarly using an optimal solution to (KPL)$^{rem}$.
\end{definition}

\begin{proposition} \label{prop:Khats}
Assuming the same value of the parameter $\alpha$ is used in (KPL)$^{all}$, (KPL)$^{rem}$, and (KPL$^p$), 
\[ \mathcal{K}^{all} ~\leq~ \mathcal{K}^* ~\leq~ \mathcal{K}^{rem}.
\]
\end{proposition}

According to Theorem \ref{thm:penalty_approx_error_bound}, the lower bounding approximation, $\mathcal{K}^{all}$, is likely to be the better choice for $\hat{\mathcal{K}}$. In this case, we can compute bounds on $|\hat{E}(\Delta,\sigma^*)|$ for various possible values of $\sigma^*$ that are independent of $\Delta$. We require the following notation.

\begin{definition}
Let $\sigma^{all} := \sum_{s \in U} c_s x^{all}_s$ for an optimal solution $(\mathbf{x}^{all}, \mathbf{y}^{all})$ to $(KPL)^{all}$. 
\end{definition}

\begin{theorem} \label{thm:more_bounds}
Suppose $\hat{\mathcal{K}} = \mathcal{K}^{all}$ in (KPL$^p$). Then undesirable locations are under-penalized, and
\[
\sigma^* - \hat{\sigma}
~\leq~ |\Delta|(1-e^{\kappa \sigma^*}) 
~\leq~ \sigma^{all}(1-e^{\kappa \sigma^*}) 
~\leq~ \sigma^{all}(1-e^{\kappa \sigma^{all}}).
\]
\end{theorem}

Next we analyze the {\bf penalty error introduced by linearization}, i.e., by approximating $e^q$, $q = -\kappa \sum_{s\in U} c_s x_s = -\kappa \sigma$, via lower-bounding tangent lines in (KPL$^t$). Let $\ddot{\sigma}$ represent the penalty applied to an optimal Kolm-Pollak score by (KPL$^t$). Thus, $\ddot{\sigma}$ is an under-approximation of $\hat{\sigma}$.

The constraints \eqref{cons:tangent_lines} and the minimizing objective function together approximate $e^q$ with the piecewise linear function, $g:[0,\sigma^{max}] \rightarrow \mathbb{R}$, defined as follows:
\[
g(q) := 
\begin{cases}
1 + q, & \text{for } q \in [0,q_{0,1}); \\
e^{\beta_i} + e^{\beta_i}(q-\beta_i), & \text{for } q \in [q_{i-1,i},q_{i,i+1}), i = 1,2,3,\dots,n-1; \\
e^{\beta_n} + e^{\beta_n}(q+\beta_n) & \text{for } q \in [q_{n-1,n},-\kappa\sigma^{max}], \\
\end{cases}
\]
where $q_{i,j}$ represents the value of $q$ for which $e^{\beta_i} + e^{\beta_i}(q-\beta_i) = e^{\beta_{j}} + e^{\beta_{j}}(q-\beta_{j}).$

In practice, the penalties, $c_s$, would often be the same, or would only take on a few different values. If the penalties are all the same (i.e., $c_s = c, \forall s \in U$), we can choose linearization points so that (KPL$^t$) is an exact formulation of (KPL$^p$).

\begin{proposition} \label{prop:no_exp_error}
If $c_s = c$ for all $s \in U$, and $\beta_i := -\kappa c i$, for $i = 0, 1, \dots, \min\{k,|U|\}$, then $\ddot{\sigma} = \hat{\sigma}$.
\end{proposition}
When it is not efficient to choose a linearization point for every possible value of $q ~=~ -\kappa \sum_{s \in U} c_s x_s$ (i.e., when there are more than a few unique penalty values), we would like to choose linearization points $\beta_i$ so that the approximation error is small. Tangent lines do not closely approximate exponential functions in general (Proposition \ref{thm:max_ex_error} and Corollary \ref{cor:exponential_growth}). However, the linearization does provide a very close approximation in this application (Proposition \ref{thm:max_practical_exponent} and Theorem \ref{thm:penalty_error_tangent}). 

\begin{proposition} \label{thm:max_ex_error}
For $b\in \mathbb{R}$, let $L_b := e^b + e^b(x-b)$ represent the tangent line to $f(x) = e^x$ at the point $(b,e^b)$. Let $g: [a,a+w] \rightarrow \mathbb{R}$ be the piecewise linear function $g(x) := \max\{L_a(x), L_{a+w}(x)\}$. The maximum error between $e^x$ and $g(x)$ on $[a,a+w]$ is, 
\[
\max_{x \in [a,a+w]} e^x - g(x) = e^a\left(e^{\frac{we^w}{e^w-1}-1} - \frac{we^w}{e^w-1}\right).
\]
\end{proposition}

\begin{corollary} \label{cor:exponential_growth}
Let $g: [a,a+w] \rightarrow \mathbb{R}$ be the piecewise linear function $g(x) := \max\{L_a(x), L_{a+w}(x)\}$. For a fixed width, $w$, the maximum error between $e^x$ and $g(x)$ on $[a,a+w]$ grows exponentially in $a$.
\end{corollary}
Thus, if $a$ is large, choosing $w$ small enough to have a reasonably small error between $e^x$ and its lower-bounding tangent lines will introduce numerical problems in a computational setting. However, in practical instances of (KPL$^q$), the argument of the exponential expression in the penalty expression, $q = -\kappa \sum_{s \in U} c_s x_s$, is bounded above by $|\epsilon|$, so that we can achieve a very close approximation to $e^q$ with a numerically stable choice of $w$. (Recall that $\epsilon$, the aversion to inequality parameter, is negative and typically takes values between $-\frac{1}{2}$ and $-2$.)
 
\begin{proposition} \label{thm:max_practical_exponent}
Let $D = \{d_r: r \in R\}$ be the distribution of distances that is used to calculate $\alpha$, and let $\mu_D$ represent the population weighted mean of $D$, $\mu_D := \frac{1}{T}\sum_{r \in R} p_r d_r$.
If $\sigma_{max} \leq \mu_D$, then $q = -\kappa \sum_{s\in U} c_s x_s \leq |\epsilon|$.
\end{proposition}
\begin{remark}
It is not guaranteed that the hypothesis of Proposition \ref{thm:max_practical_exponent} holds in general, but it is a reasonable assumption in practical instances of (KPL$^t$). The hypothesis holds if $\max_{s\in U} c_s \leq \frac{\mu_D}{\min\{k,|U|\}}$. In order for location $s$ to have a chance of being selected by the model, penalty $c_s$ must be much, much less than the anticipated optimal Kolm-Pollak score, which can be roughly approximated by $\mu_D$: $c_s << \mathcal{K}^* \approx \mu_D$. Moreover, it is likely that $k$, the number of locations to optimally place, and/or, $|U|$, the number of locations to penalize, are fairly small.
\end{remark}

Proposition \ref{thm:max_practical_exponent} indicates that the exponential expression $e^q$ can be closely approximated by tangent lines in (KPL$^t$) and we can derive bounds directly on the error in the Kolm-Pollak penalty that is introduced by the tangent line approximation.  We assume that the linearization points are equally spaced at a width of $w < 1$ apart. For context with regards to inequality \eqref{eq:tangent_bound_cases}, $A(0.01) \approx 10^{-5}$ and $A(0.001) \approx 10^{-7}$.

\begin{theorem} \label{thm:penalty_error_tangent}
For  $w \in (0,1)$, suppose $\beta_i = iw$ for $i = 0, 1, \dots, n$, where  $nw > -\kappa \sigma^{max}$. Then $\ddot{\sigma} \leq \hat{\sigma}$ and
\begin{align}
|\ddot{E}(\boldsymbol{\beta},\sigma^*)| 
~:=~ \hat{\sigma} - \ddot{\sigma}
&~\leq~ 
\frac{1}{\kappa}\ln\left(1 - \frac{A(w)}{1 + e^{\kappa\sigma^*}(e^{\kappa \Delta} - 1)}\right) \label{eq:exact_tangent_bound} \\
&~\leq~
\begin{cases}
\frac{1}{\kappa}\ln\left(1-A(w)\right), & \text{ if } \Delta \leq 0;  \\
\frac{1}{\kappa}\ln\left(1-1.25A(w)\right),
& \text{ if } \Delta > 0 \text{ and } \hat{\mathcal{K}} = \mathcal{K}^{rem},
\end{cases} \label{eq:tangent_bound_cases}
\end{align}
where $A(w) := e^{\frac{we^w}{e^w-1}-1} - \frac{we^w}{e^w-1}$.
\end{theorem}

The following is a direct result of Theorem \ref{thm:more_bounds}, Proposition \ref{prop:no_exp_error}, and Theorem \ref{thm:penalty_error_tangent}.
\begin{corollary}
If $\hat{\mathcal{K}} = \mathcal{K}^{all}$ and the linearization values are a fixed width of $w$ apart in (KPL$^t$), then less-desirable locations are underpenalized and the cumulative error in the penalty applied to the optimal Kolm-Pollak score satisfies
\begin{align*}
0
~\leq~
\sigma^* - \ddot{\sigma} 
&~\leq~ 
\sigma^{all}\left(1-e^{\kappa\sigma^*}\right) + \frac{1}{\kappa} \ln(1-A(w)).
\end{align*}
If, in addition, $c_s = c$, for all $s \in U$, and $w = -\kappa c$, then 
\[
0
~\leq~
\sigma^* - \ddot{\sigma} 
~\leq~
\sigma^{all}\left(1-e^{\kappa\sigma^{*}}\right).
\]
\end{corollary}

\begin{example} \textbf{(Penalty error bounds)} \label{ex:penalty}
To explore the error observed in practice, we solved four realistic instances of (KPL$^t$) using the data for Santa Rosa, California from the food desert application described in Section \ref{sec:grocery}.  In all instances, the linearization points were chosen to be a fixed width apart, $w$. We set $k=10$, $\epsilon = -1$, $\kappa = -\alpha = -6.93 \times 10^{-4}$, 
and $\hat{\mathcal{K}} = \mathcal{K}^{all} \approx 1273.5382$ m (meters). We chose $U$ to be the set of locations optimal to (KPL)$^{all}$ and $c_s = c = 4.8609 \approx \frac{\mathcal{K}^{rem}-\mathcal{K}^{all}}{10}$, for all $s \in U$, so that $\sigma^{all} = 48.609$ m. The optimal locations selected by (KPL$^t$) were the same for all instances and included two penalized locations, as displayed in Figure \ref{fig:penalized}. For all instances, $\mathcal{K}^* \approx 1280.1363$ m, 
$\Delta = \hat{\mathcal{K}} - \mathcal{K}^* \approx -6.5981$ m, and $\sigma^* = 9.7217$ m. Note that $-\kappa \sigma^{max} \approx 0.0337$, so we explore values of $w$ as large as $0.01$.

Tables \ref{ex2:tab1} and \ref{ex2:tab2} contain bounds on the error in the Kolm-Pollak penalty introduced by approximating $\hat{\mathcal{K}}$ and by the linear approximations, respectively, in (KPL$^t$). Table \ref{ex2:tab3} contains combined error bounds, as well as the actual observed separate and combined errors for each instance.

\begin{table}[ht]
\setlength{\tabcolsep}{7pt}
\centering
\begin{tabular}{r|cccc} 
$\sigma^*$ & 4.861 & 9.722 & 14.583 & 19.443 \\
\midrule
$\sigma^* - \hat{\sigma} \leq$ & 
0.163 & 0.326 & 0.489 & 0.651\\
\end{tabular}
\caption{Bounds on penalty error arising from $\hat{\mathcal{K}}$-approximation: $\sigma^{all}\left(1-e^{\kappa\sigma^*}\right)$.}
\label{ex2:tab1}
\end{table}

\begin{table}[ht]
\setlength{\tabcolsep}{7pt}
\centering
\begin{tabular}{r|ccccc} \toprule
$w$ & $-5\kappa$ & 0.0001 & 0.001 & 0.01 \\
\midrule
$\hat{\sigma} - \ddot{\sigma} \leq$ & 
$0^{\dagger}$ & $1.804 \times 10^{-6}$ & $1.805 \times 10^{-4}$ & 0.0181  \\
\bottomrule
\end{tabular}
\caption{Bounds on penalty error arising from linearization: $\frac{1}{\kappa} \ln(1-A(w))$. $^{\dagger}$The error is zero by Proposition \ref{prop:no_exp_error}.}
\label{ex2:tab2}
\end{table}

\begin{table}[ht]
\setlength{\tabcolsep}{7pt}
\centering
\begin{tabular}{r|ccccc} \toprule
$w$ & $\sigma^* - \hat{\sigma}$ & $\hat{\sigma}-\ddot{\sigma}$ & $\sigma^* - \ddot{\sigma}$ & $\sigma^* - \ddot{\sigma}\leq$ \\
\midrule
$-c\kappa$ & 0.0321 & 0 & 0.0321& 0.326 \\
0.0001  & 0.0321 & $9.70 \times 10^{-7}$& 0.0321 & 0.326 \\
0.001  & 0.0321 & $4.99 \times 10^{-5}$ & 0.0322 & 0.327 \\
0.01  & 0.0321 & $7.68 \times 10^{-3}$ & 0.0398 &  0.344 \\
\bottomrule

\end{tabular}
\smallskip

\caption{$\sigma^*-\hat{\sigma}$ is the error introduced by approximating the parameter 
$\hat{\mathcal{K}}$ (which does not depend on $w$). $\hat{\sigma}-\ddot{\sigma}$ is the error introduced by the linearization. $\sigma^*-\ddot{\sigma}$ is the combined error. The bounds for the combined error (in the last column) are given by $\sigma^{all}\left(1-e^{\kappa\sigma^*}\right) + \frac{1}{\kappa} \ln(1-A(w))$. The actual combined error is quite small in all cases (recall that the penalty units are meters).}
\label{ex2:tab3}
\end{table}

\begin{figure}[h!]
    \centering
    \includegraphics[width=\textwidth]{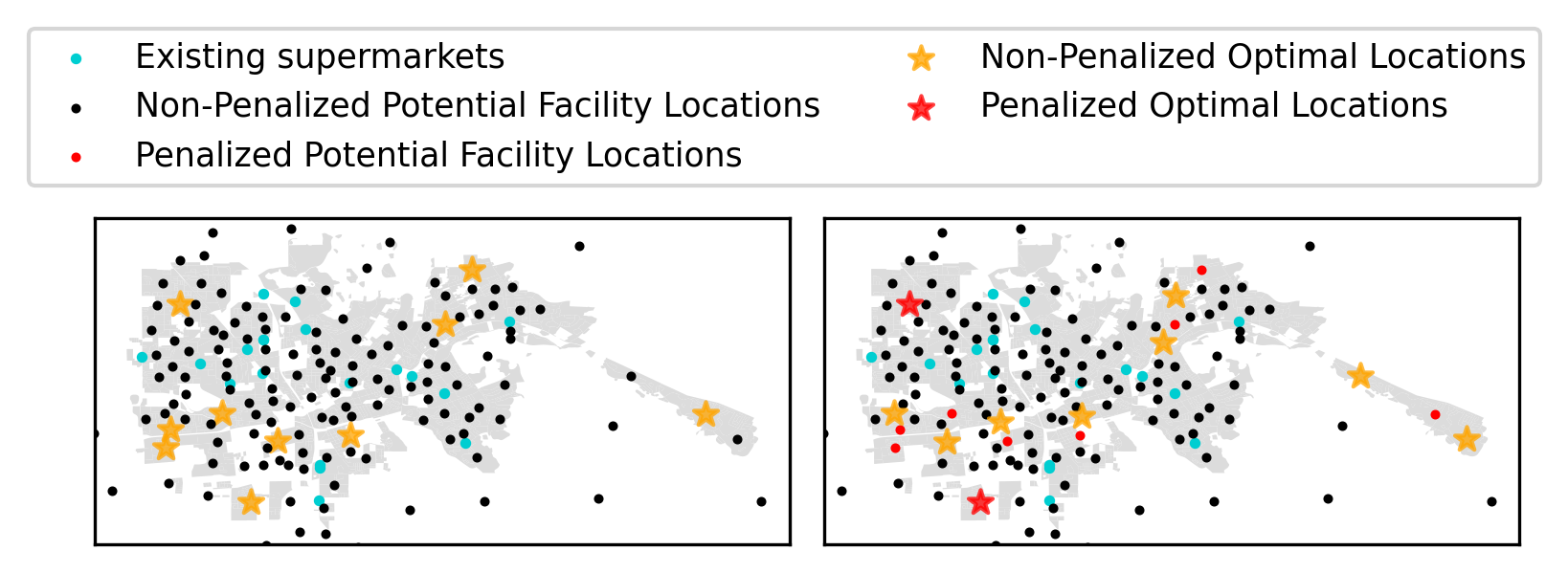}
    \caption{Optimal locations of 10 additional supermarkets in Santa Rosa, California. The figure on the left displays the solution to (KPL)$^{all}$, and the figure on the right displays the solution to (KPL)$^{p}$.}
    \label{fig:penalized}
\end{figure}

\end{example}

Note that all of the analysis in this section refers to overall penalty error, rather than per-location error. Of course, one could divide any of the error expressions or bounds by the number of locations represented by $\sigma^*$ to find a per-location error expression or bound.

\subsubsection{Penalty Strategy and Example}
We conclude this section by summarizing strategies for modeling the penalty term recommended by our analysis, then demonstrating their practical implementation. 

Choose $\hat{\mathcal{K}} = \mathcal{K}^{all}$ to keep the error associated with approximating $\hat{\mathcal{K}}$ low. A good way to scale the per-location penalty so that penalized locations are discouraged but not eliminated is to set 
\[ c_s ~=~ c ~:=~ \frac{\mathcal{K}^{rem}-\mathcal{K}^{all}}{N}, \text{ for all } s \in U,\] 
where $N$ represents the number of less-desirable locations selected by (KPL)$^{all}$. If all per-location penalties are the same ($c$), set $w := -\kappa c$ to avoid any linearization error. Otherwise, scale $w$ according to the $\Delta \leq 0$ case in inequality \eqref{eq:tangent_bound_cases} to bring the linearization error below a desired threshold.

The following case study provides a very clear example of the practical relevance of this fixed-charge penalty strategy.

\begin{example} \textbf{(Access to early voting)} 
We solved 19 instances of (KPL$^t$), the exact linear version of (KPL$^p$) with $w$ chosen as in Proposition \ref{prop:no_exp_error}, to optimally locate $k \in \{12, 13, \dots, 30\}$ early voting sites in Gwinnett County, Georgia using the method described above. Potential locations included 12 early voting locations from the 2020 and 2022 elections and 145 potential new early voting locations including community centers, libraries, churches, and fire stations, many of which were election day polling centers. We applied penalties to the 77 churches and 5 fire stations in the set of potential locations using the strategy described above, so that these locations would only be selected if a ``good enough" solution could not be found using the other types of buildings. We also required that every solution select at least 10 of the 12 existing early voting locations (moving polling locations can negatively impact turnout).

Table \ref{tab:gwinnett_penalty} summarizes the results, and Figure \ref{fig:gwinnett_penalty} visualizes the results for the case of selecting $k=20$ early voting sites. As we can see in Figure \ref{fig:gwinnett_penalty}, (KPL$^t$) obtains a very similar distribution of distances using only two penalized sites as (KPL) obtains using eight penalized sites. This is a very useful feature for decision-makers to achieve a near-optimal solution while using very few less-desirable sites. In Table \ref{tab:gwinnett_penalty}, we see a similar trend for all values of $k$. The penalty errors that arise from approximating the parameter $\hat{\mathcal{K}}$ are quite small, ranging from 0 to 2.3 meters.

\begin{figure}[ht]
\centering
\subcaptionbox{Churches and fire stations removed}[0.3\textwidth]{\includegraphics[height=1.8in]{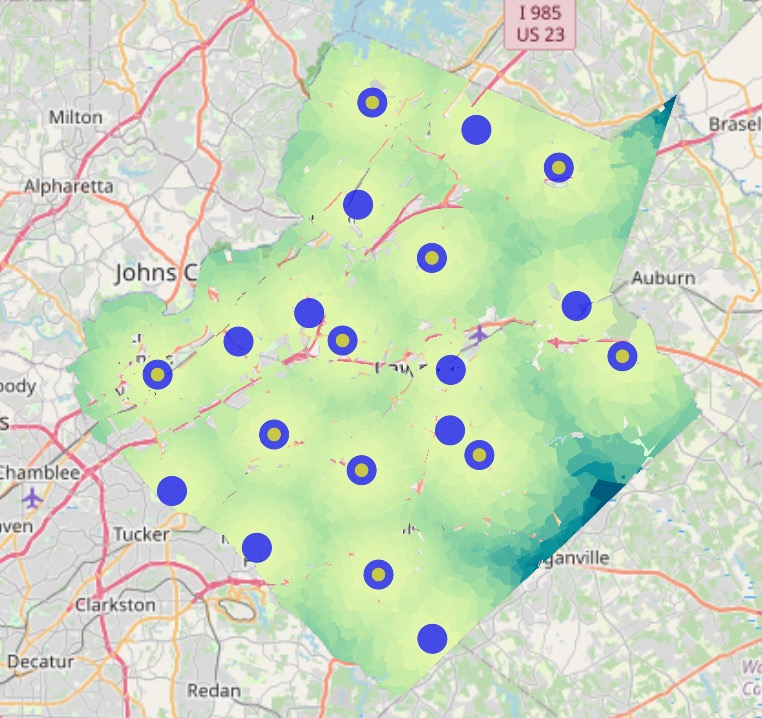}}
\subcaptionbox{All sites, no penalties}[0.3\textwidth]{\includegraphics[height=1.8in]{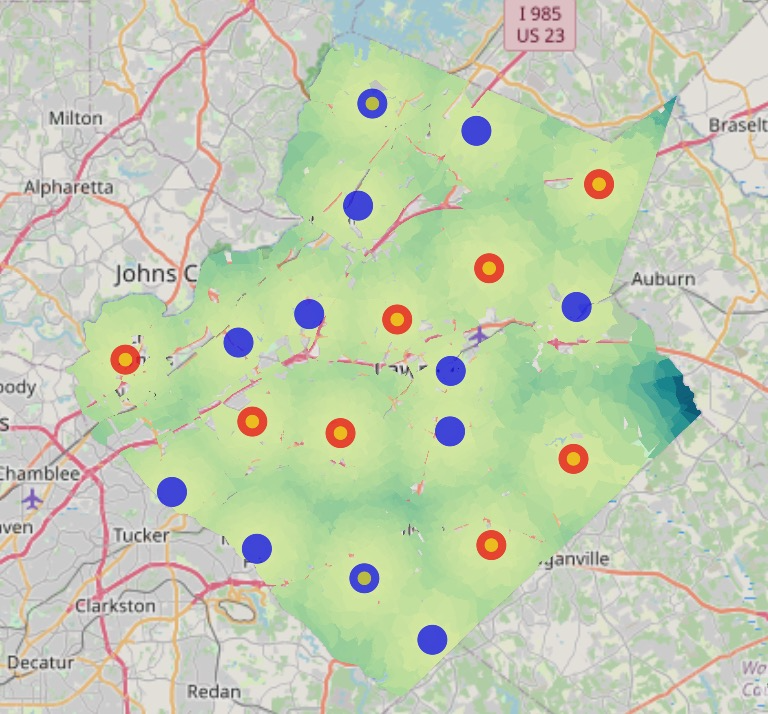}}
\subcaptionbox{All sites with penalties}[0.3\textwidth]{\includegraphics[height=1.8in]{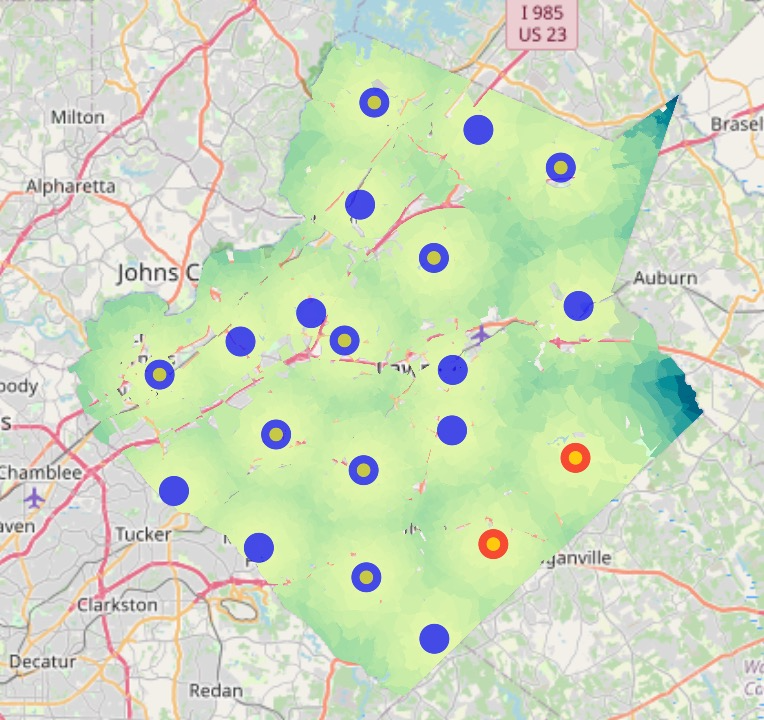}}
\caption{Gwinnett County early voting sites. Red dots represent less-suitable sites (churches and fire stations). Filled dots represent existing early voting sites. Darker background shading indicates residential areas that are farther from an early voting site. Map (a) displays the optimal solution to (KPL$^{rem}$): less-suitable sites are not included as potential locations. In map (b), eight less-suitable sites are selected by  (KPL$^{all}$): all sites are included but no penalties are applied. In map (c), the penalized model, (KPL$^t$), achieves a near-optimal Kolm-Pollak score with only 2 less-suitable sites.}
\label{fig:gwinnett_penalty}
\end{figure}

\begin{table}[ht]
\setlength{\tabcolsep}{7pt}
\centering
\begin{tabular}{c|c|cc|cc|c|cc} \toprule
& (KPL$^{rem}$) & \multicolumn{2}{|c|}{(KPL$^{all}$)} & \multicolumn{2}{c|}{(KPL$^t$)} &  & \multicolumn{2}{c}{per-site penalty} \\
$k$ & $\mathcal{K}^{rem}$ & $|P|$ & $\mathcal{K}^{all}$ & $|P|$ & $\mathcal{K}^t$ & ~gap~ & desired & error \\
\midrule
12 & 4953.8 & 2  & 4896.2 & 1 & 4925.8 & 0.01 & 28.8 & 0.0 \\
13 & 4645.5 & 3  & 4514.8 & 1 & 4515.7 & 0.00 & 43.6 & 0.0 \\
14 & 4446.4 & 3  & 4275.2 & 1 & 4285.4 & 0.00 & 57.1 & 0.2 \\
15 & 4257.0 & 3  & 4066.7 & 1 & 4084.4 & 0.00 & 63.4 & 0.4 \\
16 & 4061.9 & 4  & 3856.9 & 1 & 3876.5 & 0.01 & 51.2 & 0.4 \\
17 & 3877.2 & 4  & 3630.7 & 1 & 3658.2 & 0.01 & 61.6 & 0.6 \\
18 & 3719.7 & 5  & 3450.6 & 2 & 3480.8 & 0.01 & 53.8 & 0.6 \\
19 & 3621.3 & 6  & 3312.6 & 2 & 3369.3 & 0.02 & 51.5 & 1.0 \\
20 & 3524.0 & 8  & 3197.4 & 2 & 3262.7 & 0.02 & 40.8 & 0.9 \\
21 & 3428.0 & 7  & 3098.3 & 2 & 3188.8 & 0.03 & 47.1 & 1.5 \\
22 & 3358.5 & 6  & 3013.3 & 2 & 3123.4 & 0.04 & 57.5 & 2.2 \\
23 & 3304.3 & 7  & 2937.7 & 2 & 3064.4 & 0.04 & 52.4 & 2.3 \\
24 & 3253.7 & 9  & 2864.6 & 2 & 3009.3 & 0.05 & 43.2 & 2.1 \\
25 & 3211.7 & 12 & 2800.5 & 4 & 2889.1 & 0.03 & 34.3 & 1.0 \\
26 & 3162.5 & 12 & 2736.4 & 4 & 2841.0 & 0.04 & 35.5 & 1.3 \\
27 & 3125.7 & 12 & 2681.4 & 4 & 2784.8 & 0.04 & 37.0 & 1.3 \\
28 & 3091.5 & 12 & 2631.2 & 3 & 2775.4 & 0.05 & 38.4 & 1.9 \\
29 & 3060.7 & 14 & 2583.0 & 5 & 2667.4 & 0.03 & 34.1 & 1.0 \\
30 & 3036.4 & 15 & 2541.1 & 5 & 2620.4 & 0.03 & 33.0 & 0.9 \\
\bottomrule
\end{tabular}
\caption{The Kolm-Pollak scores ($\mathcal{K}$) and the number of fire stations and churches ($|P|$) selected by optimal solutions to three models in the Gwinnett County early voting application. Note that (KPL$^{rem}$) chooses no less-suitable sites, but leads to the highest Kolm-Pollak score. (KPL$^{all}$) achieves the best possible Kolm-Pollak scores, while (KPL$t$) achieves near-optimal Kolm-Pollak scores with fewer less-suitable sites, as measured by ``gap'': $\frac{\mathcal{K}^t-\mathcal{K}^{all}}{\mathcal{K}^{all}}$. The Kolm-Pollak scores, the per-site desired penalty, and the per-site penalty error are all measured in meters. }
\label{tab:gwinnett_penalty}
\end{table}

\end{example}

\section{Scaling the Inequality Aversion Parameter}\label{sec:scaling_parameter}


To minimize the Kolm-Pollak EDE via (KPL), we must approximate $\alpha$, the factor used to scale the inequality aversion parameter. As a consequence, we do not have complete control over the aversion to inequality represented by the optimal Kolm-Pollak score, however, we are able to precisely quantify the effect this has.  In this section, we explore computationally the impact of approximating $\alpha$ on both the aversion to inequality represented by the optimal distribution and on the optimal Kolm-Pollak score itself. Our tests include practical instances on 428 cities in which we seek to optimally locate 1, 5, or 10 additional supermarkets, as well as 5 cities in which we replace all existing polling sites with the same number of optimally located sites. 

\subsection{Iterative Calibration of $\alpha$}

Let $\epsilon_0$ represent the desired level of aversion to inequality. For each problem instance, we first estimate the optimal distribution of distances by assigning each census block to the closest existing location (grocery store or polling site). We use this distribution of distances to calculate an approximate value of $\alpha$, which we denote $\alpha_1^{in}$. Let (KPL)$_1$ represent the model (KPL) with $\kappa = \kappa_1 := \alpha_1^{in}\epsilon_0$. Let $\alpha_1^{out}$ denote the value of $\alpha$ associated with the optimal solution to (KPL)$_1$.  We can calculate the aversion to inequality represented by the optimal solution to (KPL)$_1$ as $\epsilon_1 := \kappa_1/\alpha_1^{out}$. Comparing $\epsilon_0$ and $\epsilon_1$ is a useful way to check the accuracy of $\alpha_1^{in}$. The closer $\alpha_1^{in}$ approximates $\alpha$, the closer $\epsilon_1$ will be to $\epsilon_0$. 

It is reasonable to expect that the distribution of distances associated with an optimal solution to (KPL)$_1$ will produce a more accurate $\alpha$-estimate. Thus, we solve (KPL) again, this time with $\kappa = \kappa_2 := \epsilon_0\alpha_{2}^{in}$, where $\alpha_{2}^{in} := \alpha_1^{out}$. We denote this version of (KPL) as (KPL)$_2$. We compute the aversion to inequality associated with the optimal solution to (KPL)$_2$ as $\epsilon_2 := \kappa_2/\alpha_2^{out}$, where $\alpha_2^{out}$ is calculated using the optimal distribution of distances obtained by solving (KPL)$_2$. 

In extensive computational tests, we found that the improved approximation, $\alpha_2^{in}$, was extremely accurate (results provided below). However, we also want to understand how important it is tune the $\alpha$ parameter in this way.  In other words, we want to test how sensitive optimal solutions are to the quality of the $\alpha$ approximation to determine if it is worth the extra computation time to solve (KPL) twice to obtain and use the very accurate $\alpha_2^{in}$. To do this, we used optimal solutions to (KPL)$_2$ as a benchmark to test the quality of optimal solutions to (KPL)$_1$. We denote the Kolm-Pollak EDE associated with an optimal solution to (KPL)$_n$ as $\mathcal{K}_n^*$, for $n \in \{1,2\}$. $\mathcal{K}_n^*$ represents the ``true'' value of the metric, calculated exactly using $\kappa = \epsilon_0 \alpha_n^{out}$. We find that $\mathcal{K}_1^*$ very closely approximates the more accurate $\mathcal{K}_2^*$ (detailed results below).

\subsection{$\alpha$-Approximation Experiments: Grocery Stores}

As described in detail in Section \ref{sec:computational_study}, the bulk of our computational experiments involved optimally locating 1, 5, and 10 additional grocery stores in the 500 largest cities in the United States. We completed two sets of experiments, one using a high-performance computing cluster and another with a personal laptop (specifications given in Section \ref{sec:compenvironments}). For the $\alpha$-approximation analysis described here, we used the laptop, which did not have sufficient memory to load the data for the largest cities. Thus, this study includes the 428 cities for which the laptop was able to solve (KPL)$_1$ and (KPL)$_2$ for 1, 5, and 10 additional stores.

Figure \ref{fig:epsilon} presents the distributions of $\epsilon_1$ (top row) and $\epsilon_2$ (bottom row) for the 428 cities when locating 1, 5, and 10 new stores. We see that $\epsilon_1$ has a skewed Normal distribution with a mean value that represents less aversion to inequality than $\epsilon_0 = -1$ (which is indicated by the red vertical line). As the number of new stores increases, the center of the $\epsilon_1$ distribution moves farther from the desired value of $-1$. This shift makes sense intuitively because adding more new stores naturally results in greater changes to the distribution of distances as compared to the original distribution, which was used in the original $\alpha$ approximation, $\alpha_1^{in}$. 

\begin{figure}[ht]
\centering
\includegraphics[width= \textwidth]{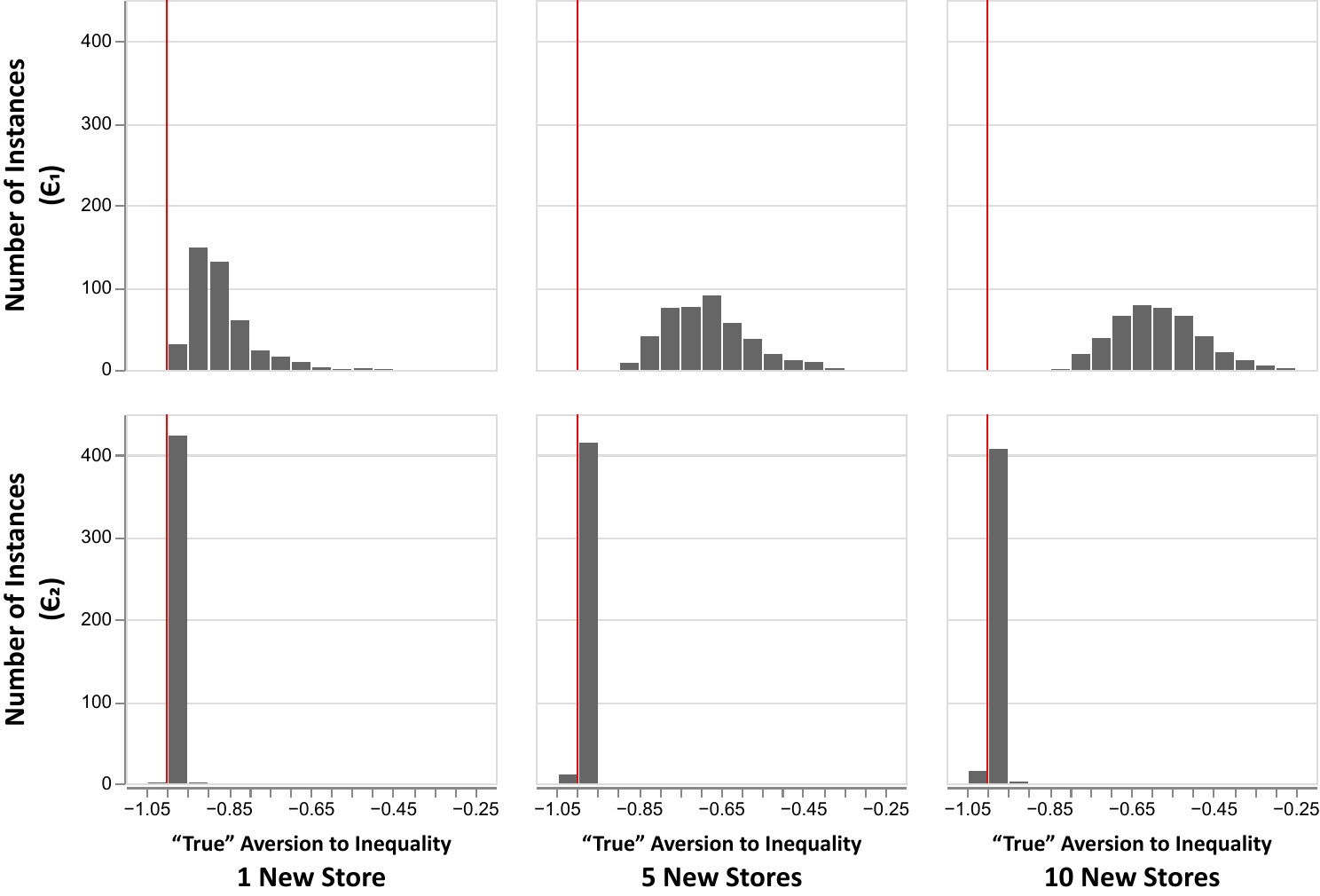}
\caption{Distribution of the ``true'' value of $\epsilon$, after the initial run ($\epsilon_1$)  and after the second run ($\epsilon_2$),  for $k=1, 5, \text{ and } 10$ new stores placed in 428 US cities.  The desired value is $\epsilon=-1$ and is indicated by the red vertical line.}
\label{fig:epsilon}
\end{figure}

As expected, $\alpha_1^{out}$ is a very good approximation, and $\epsilon_2$ very closely approximates $\epsilon_0$ in every instance. We can conclude that (KPL)$_2$ achieves a high quality optimal solution with respect to the desired level of inequality aversion. In order to test the quality of the solutions achieved by (KPL)$_1$, we calculated the raw gap between $\mathcal{K}_1^*$ and $\mathcal{K}_2^*$ as $\Delta := |\mathcal{K}_1^* - \mathcal{K}_2^*|$ meters, and the percent gap as {\sc gap} $:= \Delta/\mathcal{K}_2^*$. These results are visualized in Figure \ref{fig:gap_grocery_stores}, where the chart displays the cumulative distribution of {\sc gap} for the 428 cites, separated out by the number of additional grocery stores. 

\begin{figure}[ht]
\centering
\includegraphics[width=0.6\textwidth]{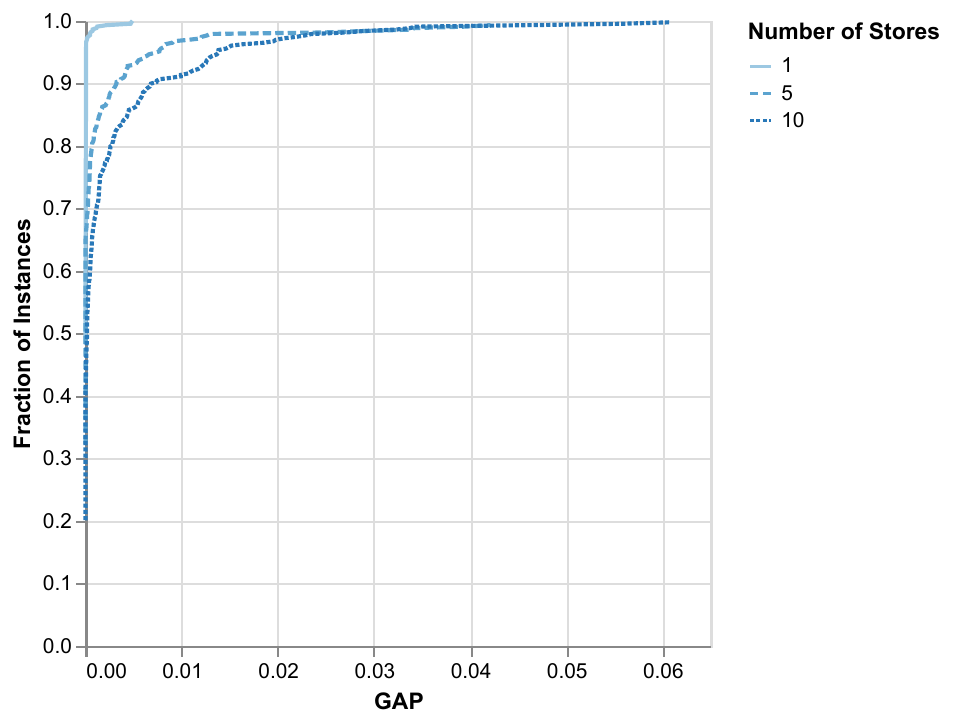}
\caption{The fraction of grocery store instances with {\sc gap}$:= |\mathcal{K}_1^*-\mathcal{K}_2^*|/\mathcal{K}_2^*$ at or below a given value. In more than 90\% of instances with $k=10$ stores, the original $\alpha$-approximation resulted in an optimal Kolm-Pollak EDE that is within 1\% of the more accurate $\mathcal{K}_2^*$. The gap is only rarely has more than 2\%. As expected, the results are better for $k=1$ and $k=5$.}
\label{fig:gap_grocery_stores}
\end{figure}

\subsection{$\alpha$-Approximation Experiments: Election Polling Sites}

For the grocery store instances, the accuracy of $\alpha_1^{in}$ decreases as the number of stores increases because the distribution of distances changes more as more stores locations are added. However, $\alpha_2^{in}$ remained a very good approximation, even as the number of stores increased. Moreover, (KPL)$_1$ still resulted in high quality optimal solutions in almost all cases, in spite of the relatively poor performance of $\alpha_1^{in}$. To more extensively test our findings, we consider situations where the optimal distributions of distances is even further from the original distances.

For this analysis, we introduce another set of practical scenarios. In these problems, we seek to optimally relocate the same number of election polling sites that were used in either the 2016 or the 2020 presidential election in five large US cities. Data for the polling scenarios is summarized in Table \ref{tab:pollingproblems}. $|R|$ indicates the number of census blocks (residential areas), $|S|$ indicates the number of potential polling sites, and $k$ indicates the number of polling sites selected (which matches the number of polling sites used in the given election year). The potential polling sites have varying capacities so we use the (KPL) model with constraints \eqref{cons:capacity}.  We set the desired aversion to inequality, denoted $\epsilon_0$, equal to $-2$ for all scenarios. 

In the polling scenarios we expect the initial $\alpha_1^{in}$ to be a less-accurate approximation than in the grocery store instances because we allow all of the existing sites to move.  This means that {\it every} resident may have a different distance in the new distribution. In contrast, with the grocery store examples, we do not remove any existing sites, so the travel distance for many residents will remain unchanged in the optimal distribution. We also expect the average distance, and the Kolm-Pollak EDE, to improve significantly from the original distribution to the optimal distribution, and this was indeed the case -- we saw that the optimal Kolm-Pollak EDE represented an improvement by more than a mile in each of the five cities. Despite these potential concerns, we saw very similar behavior for the polling instances as we did for the grocery store instances.

In Table \ref{tab:alpha0}, we display all three $\alpha$ approximations and the ``true'' aversion to inequality values for (KPL)$_1$ and (KPL)$_2$. In Table \ref{tab:alphasensitivity}, we present the raw and percent gaps between $\mathcal{K}_1^*$ and $\mathcal{K}_2^*$. In Table \ref{tab:alpha0}, we see that $\epsilon_1$ ranges between $-1.56$ and $-0.55$.  The desired level of aversion ($-2$) is not maintained by any of the optimal solutions. However, in spite of the differences between $\epsilon_1$ and the much more accurate $\epsilon_2$ (Table \ref{tab:alpha0}), the corresponding optimal Kolm-Pollak scores differ by no more than 1\% in any of the five scenarios (Table \ref{tab:alphasensitivity}).

\begin{table}[ht]
\setlength{\tabcolsep}{7pt}
\centering
\begin{tabular}{llccccr} \toprule
Scenario & City & Election & $|R|$ & $|S|$ & $k$ & Number of Variables \\ \midrule
1 & Salem & 2016 & 3102 & 92 & 8 & 285,476 \\
2 & Richmond & 2016 & 6208 & 211 & 65 & 1,310,099 \\
3 & Atlanta & 2016 & 8364 & 387 & 126 & 3,237,255 \\
4 & Cincinnati & 2016 & 7670 & 517 & 256 & 3,965,907 \\
5 & Baltimore & 2020 & 9238 & 901 & 201 & 8,324,339 \\
\bottomrule
\end{tabular}
\smallskip

\caption{A summary of the five polling scenarios used to test the sensitivity of the (KPL) model to $\alpha$ approximations.}
\label{tab:pollingproblems}
\end{table}

\begin{table}[ht]
\setlength{\tabcolsep}{7pt}
\centering
\begin{tabular}{ccccccc}
\toprule
Scenario & $\epsilon_0$ & $\epsilon_1$ &
$\epsilon_2$ & $\alpha_1^{in}$ & $\alpha_1^{out} = \alpha_2^{in}$ & $\alpha_2^{out}$   \\
\midrule
1 & -2 & -1.56 & -2.00 & 0.000192 & 0.000246 & 0.000246 \\
2 & -2 & -0.90 & -2.00 & 0.000212 & 0.000474 & 0.000474 \\
3 & -2 & -0.74 & -1.98 & 0.000179 & 0.000485 & 0.000488 \\
4 & -2 & -0.55 & -2.01 & 0.000171 & 0.000616 & 0.000614 \\
5 & -2 & -1.42 & -2.01 & 0.001277 & 0.001797 & 0.001789 \\
\bottomrule
\end{tabular}
\caption{Exact and approximate values of $\alpha$ and the aversion to inequality, where $\epsilon_0$ is the ``desired'' aversion to inequality. The aversion to inequality associated with the second iteration of the model, $\epsilon_2$, is very accurate for these instances.}
\label{tab:alpha0}
\end{table}

\begin{table}[h!t]
\setlength{\tabcolsep}{7pt}
\centering
\begin{tabular}{cccccc}
\toprule
Scenario & $\mathcal{K}_1^*$ & $\mathcal{K}_2^*$ & $\Delta$ & {\sc gap} \\
\midrule
1 & 6541.81 & 6541.81  & 0.00 & 0.0000 \\
2 & 2256.18 & 2259.49 & 3.31 & 0.0015 \\
3 & 2198.29 & 2191.37 & 6.92 & 0.0032 \\
4 & 5585.93 & 5571.38 & 14.54 & 0.0026 \\
5 & 571.21 & 566.15 & 5.06 & 0.0089 \\
\bottomrule
\end{tabular}
\caption{KP EDEs obtained via (KPL)$_1$ are close to the best-know KP EDEs, which were obtained via (KPL)$_2$. Distances in columns 2, 3, and 4 are measured in meters. }
\label{tab:alphasensitivity}
\end{table}

\subsection{Summary}

These experiments indicate that typically, optimal solutions are not highly sensitive to minor variations in $\alpha$, or equivalently, to varying levels of aversion to inequality. However, some practitioners may wish to have more control over the aversion to inequality represented by the optimal solution.  To do this, our experiments suggest the following strategy: solve (KPL) and calculate the true aversion to inequality that the optimal solution represents.  If the accuracy is not as desired, solve the model again using the distance distribution from the first solve to tune the $\alpha$ parameter.

\section{Computational Study} \label{sec:computational_study}
\subsection{Experimental Design}
We performed computational experiments to evaluate our approach to optimizing the Kolm-Pollak EDE on large-scale data sets in terms of both computational tractability and solution characteristics.  Specifically, we compare the (KPL) model with the p-centdian model (p-Ctdn), the p-median model (p-Med), and the p-center model (p-Ctr).  We chose these models as our benchmarks for the following reasons:
\begin{itemize}
    \item (p-Ctdn) represents a measure that explicitly balances efficiency with equity and still retains hope of scaling to the size instances we wish to solve.  While there are many other metrics that account for this balance, it is clear they are not appropriate for even our moderate-sized instances (recall the detailed discussion in Section \ref{background}). 
    \item While only accounting for efficiency and not equity, (p-Med) is known to scale well to large models. We use this as a benchmark on the solve time of our models. 
    \item Furthermore, (p-Med) and (p-Ctr) may be viewed as the extreme cases of both (p-Ctdn) and (KPL), and a goal of both models is to appropriately balance the two.  We consider optimal solutions to these extremes so that we may evaluate how well (p-Ctdn) and (KPL) balance the corresponding efficiency and fairness measures.
\end{itemize}

\subsubsection{Data}
\label{sec:grocery}
To test these models at scale, we drew from a real-world transportation application: optimizing access to grocery stores across the 500 largest U.S. cities.  In \cite{horton2024hundreds}, we use  this data to answer strategic questions for specific cities, 
such as: given a limit of $k$ new stores, where should they be located to maximize equitable access? And what is the minimum number of new stores needed to achieve the U.S. average level of access? In contrast, here, we demonstrate the methodological aspects of solving these large-scale equitable facility location problems.  In doing so, we are able to extensively test our model on diverse urban contexts and problem sizes.  Moreover, the only aspect of our data that is specific to supermarket access, as opposed to locating any general amenity, is the locations we chose to fix open.

 For each city, we incorporate:
\begin{itemize}
    \item Census Block locations and populations from the 2020 U.S. Census (the Census Block is the highest spatial resolution that the U.S. Census reports data);
    \item existing supermarket locations from OpenStreetMap \citep{OpenStreetMap};
    \item potential new locations at Census Block Group centroids (a Block Group is the 2nd highest spatial resolution reported by the U.S. Census);
    \item walking network distances calculated using the Open Source Routing Machine \citep{OSRM}.
\end{itemize}

\subsubsection{Computing Environments}
\label{sec:compenvironments}

To fully explore the scalability and solution quality of our model we performed two distinct sets of experiments.  In the first, we used a 2019 Dell XPS 13 9380 laptop with the Intel Core i7-8565U-8 chip and 16GB RAM to demonstrate performance on standard computing hardware.  Here, our goal was to compare computational time and the maximum size of solvable instances when using (KPL) with that achievable by the p-centdian approach.

We optimized the locations of $k \in \{1,5,10\}$ new facilities in each of the 500 cities while treating existing facilities as fixed open. We sought to solve all instances on four models: the (KPL) model with $\epsilon= -1$, and (p-Ctdn) with $\gamma \in \{ 0.1,0.5,0.9 \}$.  We chose to consider only one value of $\epsilon$ as preliminary experiments showed very little variation in the solve time for values of $\epsilon$ within the commonly-accepted interval $[-2, -0.5]$.

Our dataset included some very large optimization models, with the largest corresponding to New York City, which included 30,094 census blocks and 8,274 destinations, resulting in over 249 million binary variables in the (KPL) models. In Table \ref{tab:city_data} we summarize the problem data for five cities of interest. The first three cities correspond to the largest problem instances in the data set, as indicated by the first column which reports the rank of each city (among all 500 cities) according to the number of variables in the (KPL) models for that city. At rank 61 with around 1.5 million variables, Glendale, AZ is the highest ranked city for which all three (KPL) models (for $k \in \{1,5,10\}$) could be solved on the laptop. At rank 136 with around 0.5 million variables, Garden Grove, CA is the highest ranked city for which all nine (p-Ctdn) models (for $\gamma \in \{0.1, 0.5, 0.9\}$ and $k \in \{1,5,10\}$) could be solved on the laptop.

To further test the (KPL) model without the memory restrictions of a standard laptop, we ran all 500 cities (1,500 instances of the model) with $\epsilon=-1$ on a high-performance computing cluster with an AMD 7502 CPU processor, allocating one of the 64 available cores and 512 GB memory to each instance. The New York City instances required additional memory and were solved using 2 TB RAM.  In this setting, we were able to solve (KPL) on every instance for all 500 cities to optimality.  We also solved all instances on both the (p-Med) and (p-Ctr) models using the cluster.

In all cases, we implemented the models in Python using the optimization modeling language Pyomo \citep{hart2011pyomo,bynum2021pyomo} and solved them using the mixed-integer linear programming solver Gurobi \citep{gurobi}, using the default optimality gap of 0.01\%.

\begin{table}[ht]
\setlength{\tabcolsep}{7pt}
\centering
\begin{tabular}{clrrrr}
\toprule
Rank& City & Population  &|R| & |S| &  Number of Variables \\
\midrule
1&New York, NY & 8,784,592 & 30,094 & 8,274  &  249,006,030 \\
2&Los Angeles, CA & 3,849,235 & 23,038 & 4,701  & 108,306,340 \\
3&Chicago, IL & 2,733,239 & 33,011  & 3,237 &  106,859,845 \\
61 &Glendale, AZ & 241,320 & 2,249 & 706   &  1,588,500 \\
136 &Garden Grove, CA & 171,455 & 951 & 608   & 578,816 \\
\bottomrule
\end{tabular}
\caption{Summary of data from the three largest cities in the food deserts study, plus the largest instances solved on a standard laptop for (KPL) and (p-Ctdn) respectively.}
\label{tab:city_data}
\end{table}

\subsection{Computational Results: Scalability Analysis}\label{Solve Time Comparison}
\subsubsection{Comparison with p-centdian Using a Personal Laptop}

The personal laptop solved all nine (p-Ctdn) models for 348 cities and all three (KPL) models for 428 cities. There were an additional six cities where (KPL) was solved for $k=1$, but not for $k=5$, $k=10$, or both.  
The laptop encountered memory errors when attempting to load models corresponding to the other 66 cities. There were also many instances (86 cities) for which the laptop was able to load the data and the (p-Ctdn) model, but the kernel crashed during presolve due to memory issues.

\autoref{fig:solve_times_food_new} presents performance profiles showing the solve times for solved instances. The horizontal axis is on a logarithmic scale and displays the solve time (in seconds); this corresponds to the ``wall time'' for each instance, i.e., the length of time between when the solver is called and when a solution is returned. The vertical axis shows the number of cities that were solved to optimality within the given number of seconds.  We observe that (KPL) out-performs (p-Ctdn) by around a factor of 10 for all values of $k$ and $\gamma$.

\begin{figure}[ht]
\centering
\includegraphics[width=\textwidth]{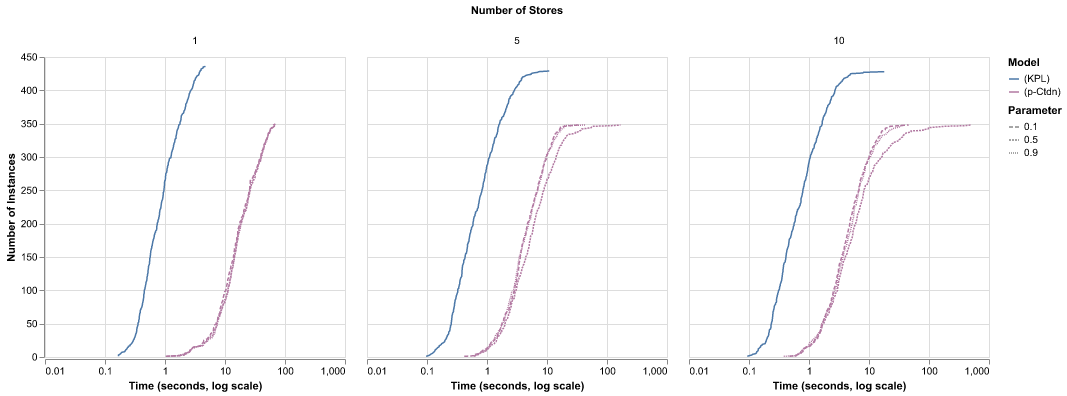}
\caption{Performance profiles showing food desert application solve times on a laptop for (p-Ctdn) and (KPL) on 428 U.S. cities for $k \in \{1,5,10\}$ additional supermarkets. (The parameter is $\gamma$ and refers to the ($p$-Ctdn) instances). The shorter lines for the (p-Ctdn) model indicate that fewer instances of that model solved on the laptop, and the position of the lines indicate that for the problems that (p-Ctdn) succeeded at solving, it was around 10 times slower than (KPL).} 
\label{fig:solve_times_food_new}
\end{figure}

\subsubsection{Comparison with p-median and p-center Using a High-Performance Computing Cluster}

To overcome the memory limitations on a standard laptop, we solved (KPL) with $\epsilon=-1$ on the high-performance computing cluster for all 500 of the largest US cities.  Here, we compare solve time with (p-Med), i.e., (p-Ctdn) with $\gamma=0$.  This represents the conventional approach focused on efficiency, which we know scales well to large instances. At the other end of the spectrum, we compare solve time with (p-Ctr), i.e., (p-Ctdn) with $\gamma=1$.

Observing the performance profiles in Figure \ref{fig:solve_times_food_new2}, we see that (p-Med) and (KPL) achieved similar solve times across all instances, 
and this performance parity persisted as the number of new stores increased. As expected, (KPL) typically solved significantly faster than (p-Ctr).  The New York City instances, detailed in \autoref{tab:nyc}, provide specific insights into relative performance. While (KPL) took longer than (p-Med) in all cases, the gap was modest, especially as the number of additional stores increased -- less than 0.05\% difference for $k=5$ and $k=10$. 
Both models substantially outperformed (p-Ctr), which required approximately twice the solve time.

\begin{figure}[ht]
\centering
\includegraphics[width=\textwidth]{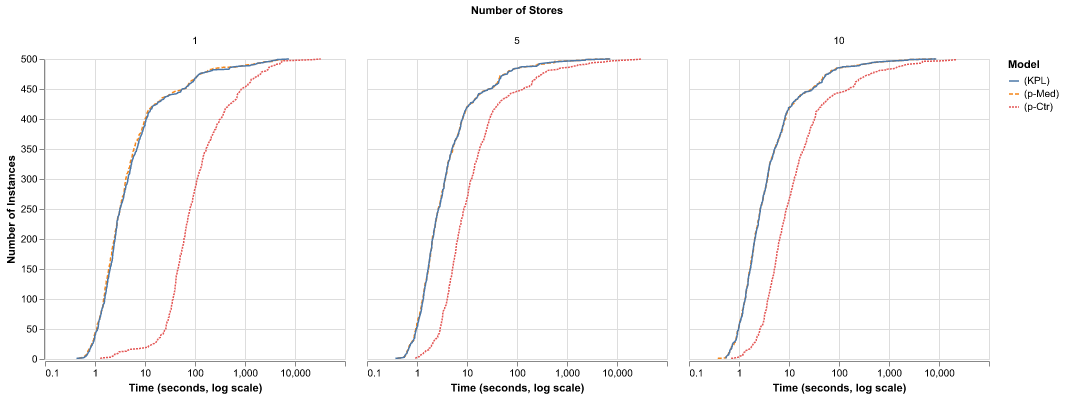}
\caption{Performance profiles showing food desert application solve times on the High Performance Computing Cluster for (KPL), (p-Med) and (p-Ctr) on all 500 U.S. cities for $k=1,5,10$ additional supermarkets.} 
\label{fig:solve_times_food_new2}
\end{figure}

\begin{table}[ht] 

\setlength{\tabcolsep}{7pt}
\centering
\begin{tabular}{crrr} \toprule
$k$ & (KPL) & (p-Med) & (p-Ctr) \\ \midrule
1 & 57.3 & 38.6 & 83.5\\
5 & 122.5 & 122.4 & 231.6 \\
10 & 146.2 & 140.9 & 286.4 \\
\bottomrule
\end{tabular}
\smallskip

\caption{Solve times (in minutes) for the New York City instances.}
\label{tab:nyc}
\end{table}

\subsection{Computational Results: Properties of Optimal Solutions}\label{Solution Comparison}

Each optimal solution to a problem instance defines a distribution of distances that residents must travel to reach their assigned (in this case, nearest) supermarket.  To visually compare the solution quality of (KPL) with that of (p-Ctdn), we plot the maximum value and the mean value of the distributions given by the respective optimal solutions.  In particular, we wish to compare the mean value of each optimal distribution with the best possible mean value, obtained by solving (p-Med), and similarly, we compare the maximum value of each optimal distribution with the optimal maximum value, obtained by solving (p-Ctr).

Figure \ref{fig:compare_mean_food} provides an overview of how the mean distance traveled by residents to reach their closest grocery store in an optimal solution of (KPL) compares with optimal solutions to (p-Ctr) and (p-Med).  The 500 cities are represented on the horizontal axis in each of the three charts, which display scenarios with $k \in \{1,5,10\}$ new stores. They are sorted according to the mean distance achieved by the optimal solution to (KPL), so the three colored shapes representing optimal solutions to (KPL), (p-Med), and (p-Ctr) for a single city appear in a vertical line. Figure \ref{fig:pct_compare_mean_food} is similar, but compares (p-Ctdn) to (p-Ctr) and (p-Med) for $\gamma \in \{0.1,0.5,0.9\}$ for the 348 cities for which we have (p-Ctdn) results.

Notably, the average distance for (p-Ctr) is consistently much higher than that of both (p-Med) and (KPL), while the average distances for (p-Med) and (KPL) closely align.  This general trend holds for (p-Ctdn) with $\gamma = 0.1$, although interestingly, it seems that an even smaller value of $\gamma$ would be necessary to align with (p-Med) to the same extent as (KPL).

\begin{figure}[hbt!]
\centering
\begin{subfigure}[b]{.95\textwidth}
\includegraphics[width=\textwidth]{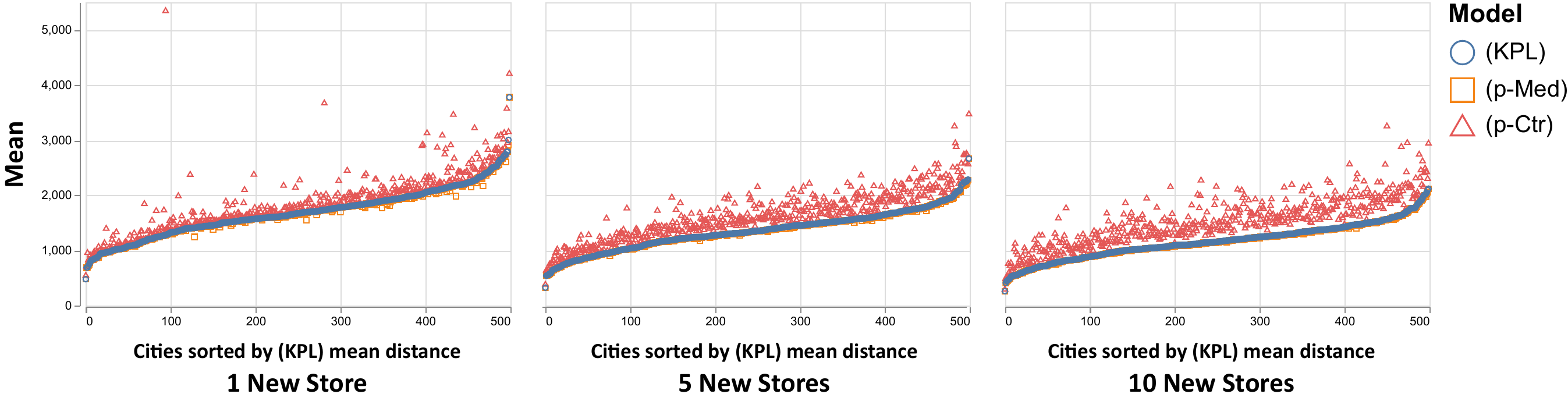}
\caption{Average walking distances for residents to nearest supermarket in optimal solutions of (p-Med), (p-Ctr), and (KPL) models for 500 largest U.S. cities.}
\label{fig:compare_mean_food}
\end{subfigure}

\begin{subfigure}[b]{\textwidth}
\includegraphics[width=\textwidth]{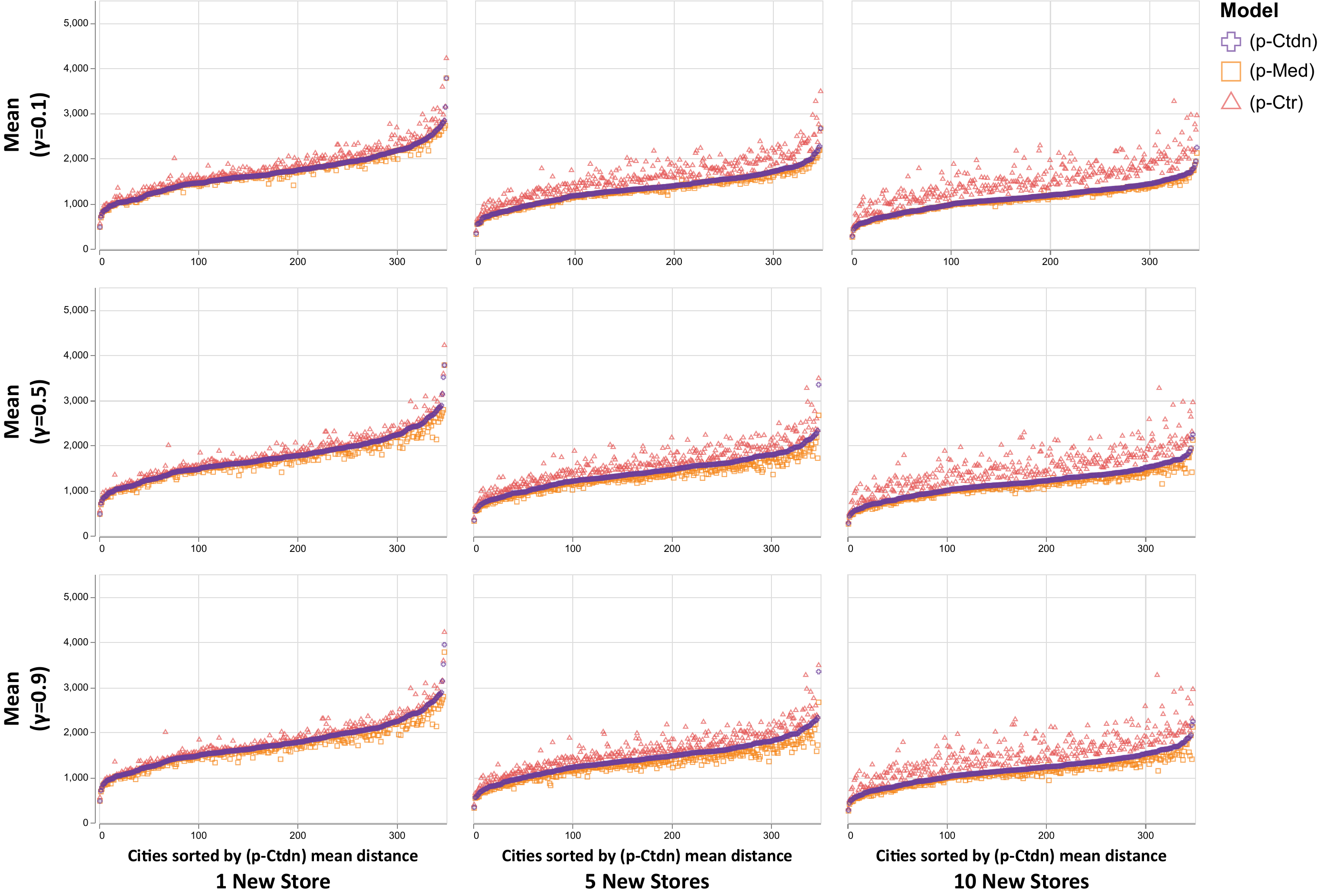}
\caption{Average walking distances for residents to nearest supermarket in optimal solutions of (p-Med), (p-Ctr), and (p-Ctdn) models for 350 largest U.S. cities.} 
\label{fig:pct_compare_mean_food}
\end{subfigure}
\caption{}
\label{}
\end{figure}

Figures \ref{fig:compare_max_food} and \ref{fig:pct_compare_max_food} are similar to Figures \ref{fig:compare_mean_food} and \ref{fig:pct_compare_mean_food} except that they present the maximum (rather than mean) distance traveled in each optimal distribution of distances. The horizontal axis is sorted according to the maximum distance in an optimal solution to (KPL) (resp. (p-Ctdn)). 

As expected, we observe that (p-Med) consistently yields significantly higher maximum distances compared to (p-Ctr), (KPL) and p-(Ctdn). While (KPL) optimal solutions correspond to greater maximum distances when compared to (p-Ctr), they still result in a substantial improvement over the maximum distances in (p-Med) optimal solutions.  For $\gamma= 0.1$, (p-Ctdn) demonstrates a similar balance, while still remaining closer to the (p-Ctr) solutions.  For larger values of $\gamma$, (p-Ctdn) aligns extremely closely with (p-Ctr).

\begin{figure}[hbt!]
\centering
\begin{subfigure}[b]{\textwidth}
\includegraphics[width=\textwidth]{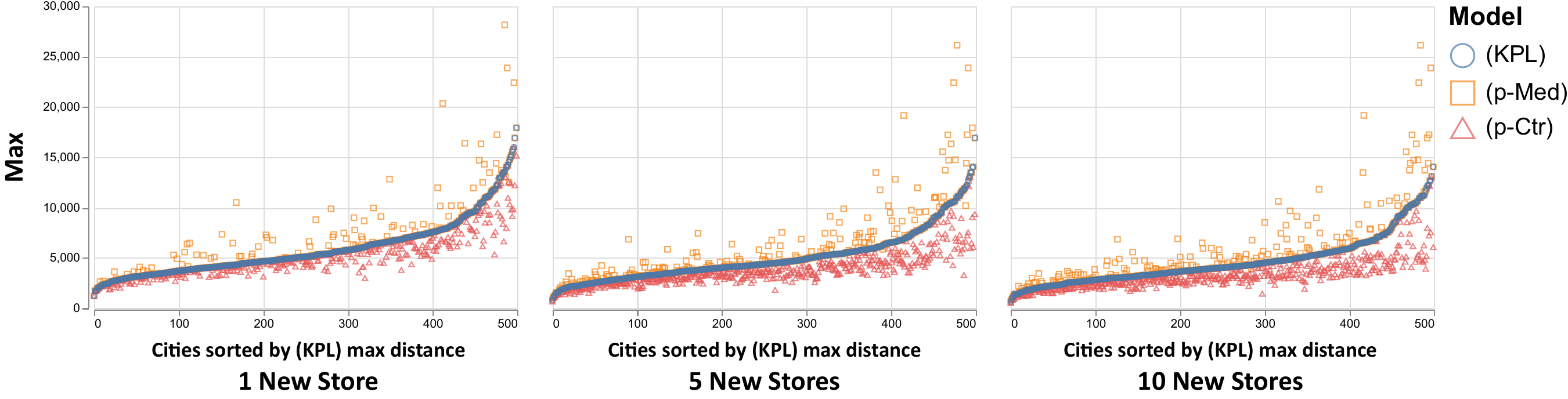}
\caption{Maximum walking distances for any resident to nearest supermarket in optimal solutions of (p-Med), (p-Ctr), and (KPL) models for 499 largest U.S. cities. Note: one outlier with a large (p-Med) maximum value was removed from the $k=1$ store chart to improve readability.}
\label{fig:compare_max_food}
\end{subfigure}
\begin{subfigure}[b]{\textwidth}
\includegraphics[width=\textwidth]{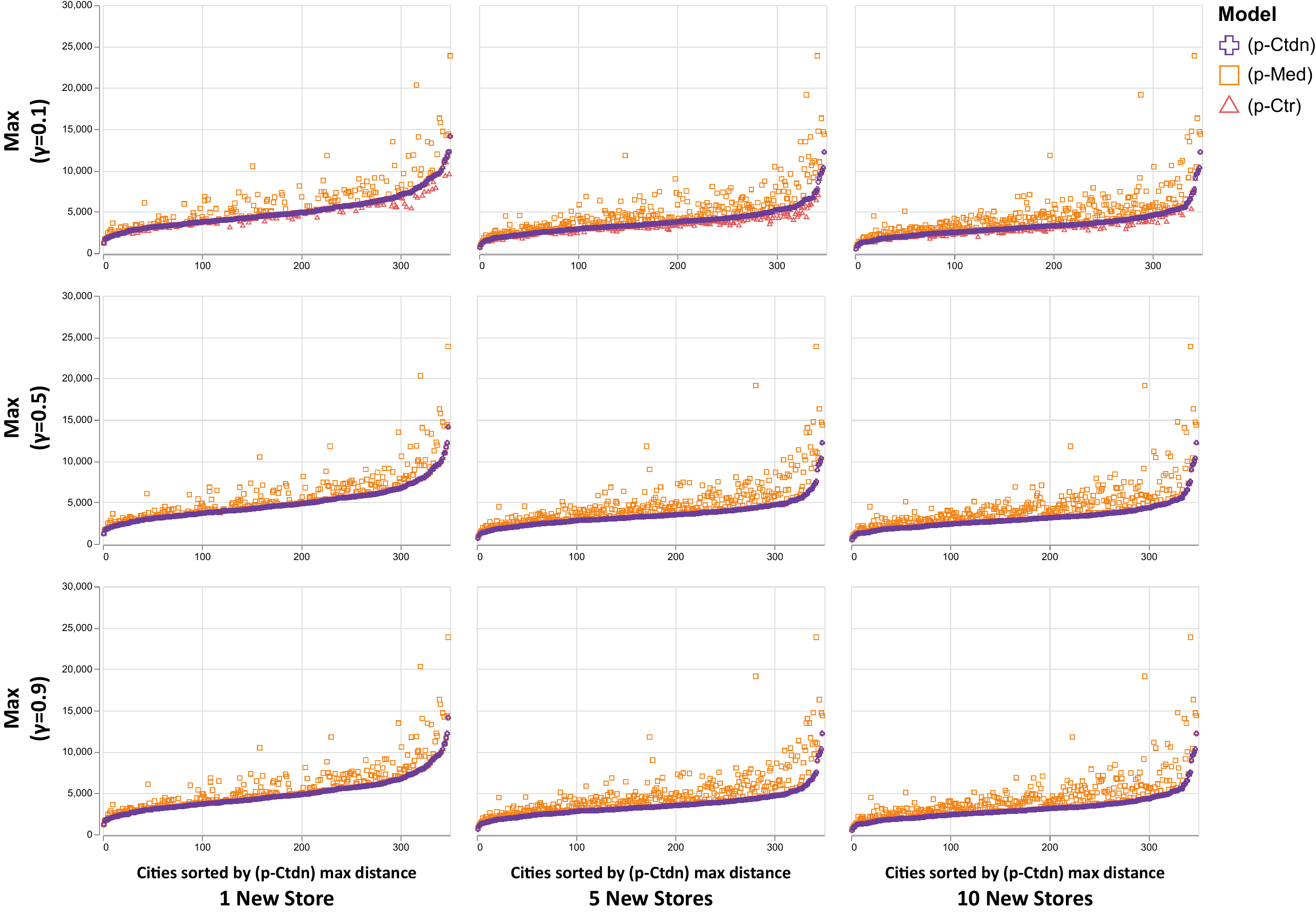}
\caption{Maximum walking distances for any resident to nearest supermarket in optimal solutions of (p-Med), (p-Ctr), and (p-Ctdn) models for 350 U.S. cities.}
\label{fig:pct_compare_max_food}
\end{subfigure}
\caption{}
\label{}
\end{figure}

We wrap up this section with a small table of statistics on the set of cities that were solved for both (p-Ctdn) and (KPL). In particular, we demonstrate the tradeoff between improving the maximum distance traveled and increasing the average travel distance versus the optimal (p-Med) solution. We can see that even with $\gamma = 0.1$, which heavily weights improving the average distance versus the maximum distance, the (p-Ctdn) model is more successful at decreasing the maximum distance, while the (KPL) is more successful at maintaining a near-optimal average experience while addressing equity across the distribution. Of course, the (p-Ctdn) could be tuned to focus more on average distance by choosing a smaller value for $\gamma$. 

\begin{table}[h!]
\setlength{\tabcolsep}{7pt}
\centering
\begin{tabular}{lrrr}
\toprule
& $k=1$  & $k=5$ & $k=10$  \\
\midrule
\multicolumn{4}{l}{\textbf{(KPL) versus (p-Med)}} \\
Average increase in mean distance (meters) & +9.63  & +5.26 & +3.08  \\
Average decrease in maximum distance (meters) & -319.26 & -376.03 & -429.45  \\
\addlinespace
\multicolumn{4}{l}{\textbf{(p-Ctdn) versus (p-Med)}} \\
Average increase in mean distance (meters) & +26.81 & +33.17 & +26.80  \\
Average decrease in maximum distance (meters) & -722.84 & -1187.79 & -1226.48 \\
\bottomrule
\end{tabular}
\caption{Comparing decrease in maximum distance versus increase in mean distance for (KPL) with $\epsilon = -1$ and (p-Ctdn) with $\gamma = 0.1$ relative to (p-Med) optimal solutions. Averages are taken over the 348 food desert cities for which the (p-Ctdn) models solved successfully.}
\label{tab:avgsupermarket}
\end{table}

\subsubsection{Differences between (KPL) and (p-Ctdn)}

We have seen that optimal solutions to both (KPL) and (p-Ctdn) can be interpreted as navigating the trade-off between efficiency and fairness by balancing the mean (p-Med) and maximum (p-Ctr) objectives.  However, it is important to note that while both metrics can be interpreted in this way, the Kolm-Pollak EDE is able to distinguish between distributions even when p-centdian is not.

Viewed like this, we see that (KPL) may in fact serve as a better proxy to the general Discrete Ordered Median Problem (DOMP) than (p-Ctdn) when the goal is to solve larger instances than can be handled using the general framework. Nuanced solutions that address equity across the entire distribution, such as those attained by specially constructed versions of the DOMP, cannot be achieved by (p-Ctdn) because the measure only takes into account the maximum and mean values of a distribution. Hence, in the context of the DOMP, all distances that are {\it not} the maximum are forced to have equal weight.  This means that in practice, the corresponding facility location problem may often have multiple optimal solutions -- two distributions with equal mean and maximum values cannot be differentiated in (p-Ctdn).  In contrast, the general DOMP is able to appropriately distinguish between distributions with equal mean and maximum values by assigning suitable weights to every distance in the distribution.  Example \ref{ex:KPLvspCtdn} demonstrates that the Kolm-Pollak EDE similarly distinguishes between such distributions in a reasonable way.  Moreover, as we have seen, (KPL) retains the desirable computational properties that allow us to solve instances much larger than if we were to use a true DOMP.

\begin{example}
\label{ex:KPLvspCtdn}
Consider a simple scenario with three residential areas, each with a population of one. We are interested in the value of the Kolm-Pollak EDE and the p-centdian metric for a family of distance distributions, the details of which are provided in Table \ref{tab:KPpCtdnEx}.

\begin{table}[h!]
   \setlength{\tabcolsep}{7pt}
\centering
\begin{tabular}{c |c | c } \toprule
Residential Area & Population  & Distance to Assigned Facility (as a function of $\Delta$)   \\ \midrule
1 & 1 & $0 + \Delta$ \\
2 & 1 & $80 - \Delta$  \\
3 & 1 & $100$ \\
\bottomrule
\end{tabular}
\smallskip

\caption{We consider a family of distance distributions, parameterized by $\Delta \in [0,80]$, for three residential areas. Note that the maximum of each distribution is equal to 100 and does not depend on $\Delta$, moreover, the mean of each distribution is equal to $180/3=60$, and also does not depend on $\Delta$. }
\label{tab:KPpCtdnEx}
\end{table}

The family is constructed such that the mean value and the maximum value do not depend on the parameter $\Delta$. Hence, the value of the p-centdian remains constant as we vary $\Delta$ for any fixed choice of $\gamma$.  However, while the mean and maximum values do not change, the family is constructed so that the equality in the distribution does vary with $\Delta$.  With $\Delta = 0$, or $\Delta = 80$, the distribution has maximum inequality, conversely, when $\Delta = 40$, the distribution is significantly more fair.  

Figure \ref{fig:KPvsPCtdn} displays the value of the p-centdian for $\gamma = 0.1, 0.25, 0.5$ and the Kolm-Pollak EDE for $\epsilon = -0.5, -1, -2$ as $\Delta \in [0,80]$ varies.  As expected,  p-centdian does not distinguish between the quality of the distributions and remains constant, whereas the Kolm-Pollak EDE is able to reflect our preference for distributing the 80 units of distance between residential areas 1 and 2 as equally as possible.

\begin{figure}[hbt!]
\centering
\includegraphics[width=0.7\textwidth]{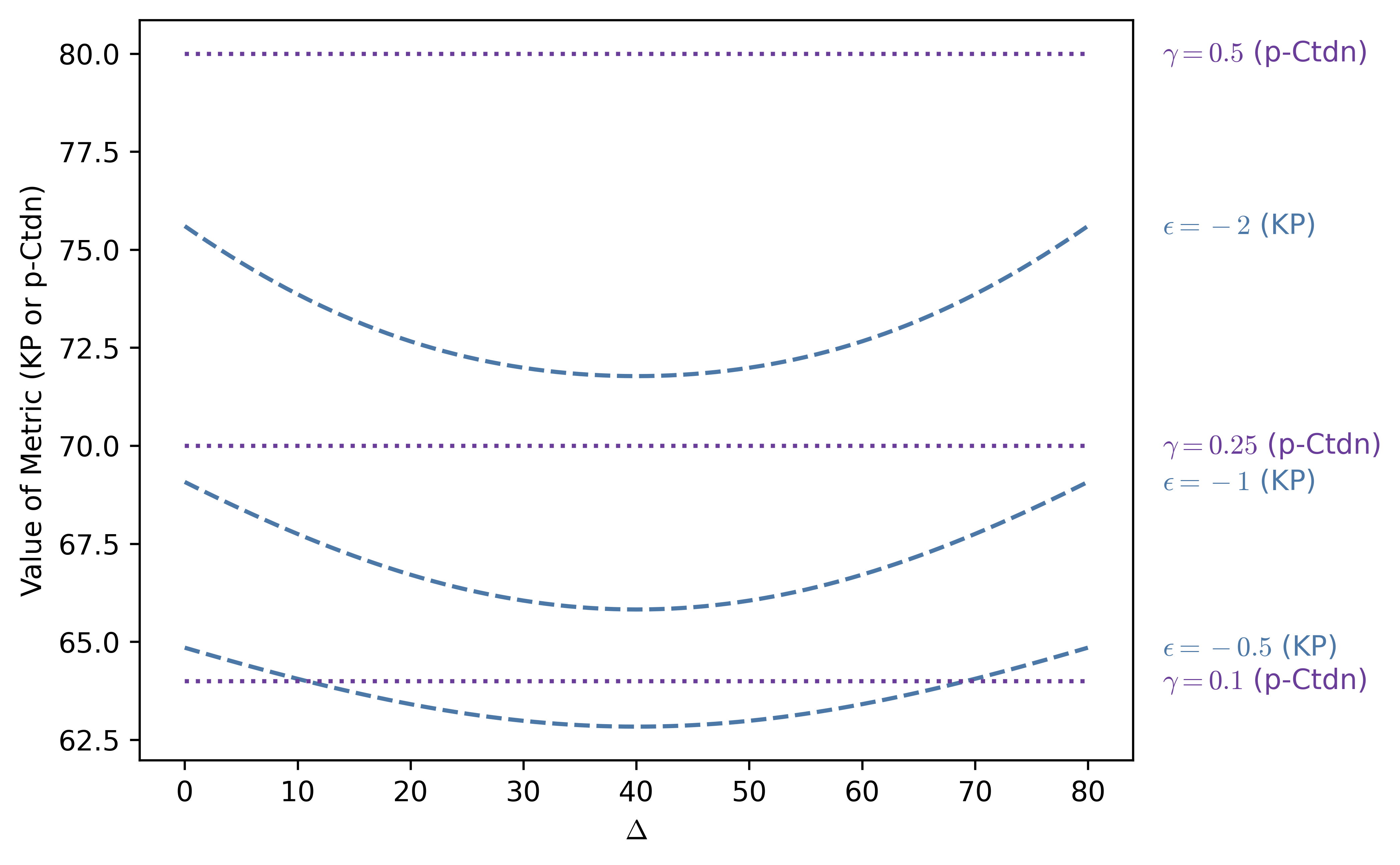}
\caption{Comparing the Kolm-Pollak and p-centdian metrics for the family of distributions as $\Delta$ varies from 0 to 80.  The Kolm-Pollak (blue, dashed) is always able to account for the difference in equality as $\Delta$ varies, whereas the p-centdian (red, dotted) is not, remaining constant for all choices of $\gamma$. }
\label{fig:KPvsPCtdn}
\end{figure}

\end{example}

\section{Discussion and Conclusions}\label{conclusion}
\label{Conclusion}

This research addresses a fundamental challenge in transportation facility location: how to optimize the Kolm-Pollak EDE while maintaining computational tractability for real-world networks. We have demonstrated that our approach not only scales to very large problems but achieves this without sacrificing solution quality. The linear proxy formulation (KPL) performs comparably to the p-median model computationally while delivering solutions that better serve both average and worst-case users. Similar to DOMP models, which do not yet scale computationally, (KPL) delivers nuanced optimal solutions, equitably distributing travel burdens within the interior of a distribution of distances, while simultaneously achieving near-optimal average welfare.

Our computational experiments revealed three methodological advances:
\begin{enumerate}
    \item \textit{The approach is computationally efficient}: The (KPL) model solved instances with over 200 million binary variables, demonstrating comparable performance to the p-median model and significantly outperforming the p-center and p-centdian approaches. This represents a significant improvement over previous equity-based approaches that became intractable for realistic problem sizes.
    \item \textit{Equality can be improved without compromising the average travel time}: As shown in Section \ref{Solution Comparison}, the model balances efficiency and equity. On average across all cities, sacrificing less than 10 meters in mean travel distance yielded improvements of over 400 meters in maximum distance. This demonstrates that incorporating equity considerations need not compromise overall system efficiency.
    \item \textit{Potential locations can be preferred or penalized}: The penalty framework we developed enables planners to incorporate location preferences and tradeoffs while maintaining the model's computational advantages. This allows for nuanced consideration of factors beyond pure distance metrics.
\end{enumerate}

The ability to efficiently optimize the Kolm-Pollak EDE presents new possibilities for incorporating equity in transportation facility location decisions. Our case studies demonstrated applications to both public services (polling locations) and private facilities (supermarkets), but the framework extends naturally to many transportation contexts:
\begin{itemize}
    \item \textit{Transit Stop Location}: Optimizing the placement of bus stops or subway stations to balance coverage with travel time equity
    \item \textit{Emergency Services}: Locating ambulance bases or fire stations to ensure both quick average response times and acceptable worst-case scenarios
    \item \textit{Freight Networks}: Designing distribution center networks that consider both efficiency and fair service levels across regions
    \item \textit{Infrastructure Investment}: Prioritizing transportation infrastructure improvements to address both system-wide performance and localized inequities
\end{itemize}

The penalty framework proves particularly valuable for transportation applications by enabling explicit consideration of cost-distance tradeoffs for facility siting; risk exposure from natural hazards or other disruptions; environmental impact considerations and; land use compatibility and community preferences.

For example, in disaster recovery and resilience planning, transportation networks and facility locations and recovery play a critical role \citep{Anderson2022-vr, Svirsko2022-ns, Logan2020-refresilience}. The Kolm-Pollak optimization approach could enhance such recovery planning by explicitly incorporating equity considerations into the reopening sequence while maintaining computational tractability. The penalty framework could be especially useful in this context, allowing planners to incorporate both repair costs and flood risk exposure when prioritizing facility restoration. The penalty framework could also support preparedness and hazard resilience efforts, as planners can consider these same aspects in initial facility location decisions. 

The significance of this work extends beyond transportation facility location. The Kolm-Pollak EDE has emerged as the preferred metric for evaluating distributions of both amenities and burdens across various domains. Recent applications include network restoration strategies \citep{10014676}, urban heat exposure \citep{Hsuetal2021}, emissions \citep{MansurSheriff2021}, and even microbial exposure in urban environments \citep{MicrobialExposure}. Our optimization approach could enhance these applications by moving from descriptive analysis to prescriptive decision support.
The penalty framework we developed is particularly promising for multi-criteria decision making in transportation and wider planning applications. 

By providing a computationally tractable approach to optimizing the Kolm-Pollak EDE, this research establishes a foundation for incorporating equity considerations more systematically in transportation facility location decisions. The demonstrated scalability to real-world problems, combined with the flexibility of the penalty framework, makes this approach immediately applicable to pressing transportation planning challenges while opening new avenues for research in equitable system design.

\section{Acknowledgements}
{\small This work used computing resources at the Center for Computational Mathematics, University of Colorado Denver, including the Alderaan cluster, supported by the National Science Foundation award OAC-2019089.}

\singlespacing
\renewcommand*{\bibfont}{\small}
\bibliographystyle{elsarticle-harv}
\bibliography{equitablefacilitylocation}

\begin{thebibliography}{53}
\expandafter\ifx\csname natexlab\endcsname\relax\def\natexlab#1{#1}\fi
\providecommand{\url}[1]{\texttt{#1}}
\providecommand{\href}[2]{#2}
\providecommand{\path}[1]{#1}
\providecommand{\DOIprefix}{doi:}
\providecommand{\ArXivprefix}{arXiv:}
\providecommand{\URLprefix}{URL: }
\providecommand{\Pubmedprefix}{pmid:}
\providecommand{\doi}[1]{\href{http://dx.doi.org/#1}{\path{#1}}}
\providecommand{\Pubmed}[1]{\href{pmid:#1}{\path{#1}}}
\providecommand{\bibinfo}[2]{#2}
\ifx\xfnm\relax \def\xfnm[#1]{\unskip,\space#1}\fi
\bibitem[{Alem et~al.(2022)Alem, Caunhye and Moreno}]{Alem2022-rp}
\bibinfo{author}{Alem, D.}, \bibinfo{author}{Caunhye, A.M.},
  \bibinfo{author}{Moreno, A.}, \bibinfo{year}{2022}.
\newblock \bibinfo{title}{Revisiting gini for equitable humanitarian
  logistics}.
\newblock \bibinfo{journal}{Socioecon. Plann. Sci.} \bibinfo{volume}{82},
  \bibinfo{pages}{101312}.
\bibitem[{Allison(1978)}]{Allison1978-lk}
\bibinfo{author}{Allison, P.D.}, \bibinfo{year}{1978}.
\newblock \bibinfo{title}{Measures of inequality}.
\newblock \bibinfo{journal}{Am. Sociol. Rev.} \bibinfo{volume}{43},
  \bibinfo{pages}{865--880}.
\bibitem[{Anderson et~al.(2022)Anderson, Kiddle and Logan}]{Anderson2022-vr}
\bibinfo{author}{Anderson, M.J.}, \bibinfo{author}{Kiddle, D.A.F.},
  \bibinfo{author}{Logan, T.M.}, \bibinfo{year}{2022}.
\newblock \bibinfo{title}{The underestimated role of the transportation
  network: Improving disaster \& community resilience}.
\newblock \bibinfo{journal}{Transp. Res. Part D: Trans. Environ.}
  \bibinfo{volume}{106}, \bibinfo{pages}{103218}.
\newblock \DOIprefix\doi{10.1016/j.trd.2022.103218}.
\bibitem[{Atkinson(1970)}]{EDEAt}
\bibinfo{author}{Atkinson, A.}, \bibinfo{year}{1970}.
\newblock \bibinfo{title}{On the measurement of inequality}.
\newblock \bibinfo{journal}{Journal of Economic Theory} \bibinfo{volume}{2},
  \bibinfo{pages}{244–263}.
\bibitem[{Barbati and Bruno(2017)}]{Barbati2018-nd}
\bibinfo{author}{Barbati, M.}, \bibinfo{author}{Bruno, G.},
  \bibinfo{year}{2017}.
\newblock \bibinfo{title}{Exploring similarities in discrete facility location
  models with equality measures}.
\newblock \bibinfo{journal}{Geographical Analysis} \bibinfo{volume}{50},
  \bibinfo{pages}{378--396}.
\bibitem[{Barbati and Piccolo(2016)}]{Barbati2016-rj}
\bibinfo{author}{Barbati, M.}, \bibinfo{author}{Piccolo, C.},
  \bibinfo{year}{2016}.
\newblock \bibinfo{title}{Equality measure properties for location problems}.
\newblock \bibinfo{journal}{Optim. Lett.} \bibinfo{volume}{10},
  \bibinfo{pages}{903--920}.
\bibitem[{Boland et~al.(2006)Boland, Domínguez-Marín, Nickel and
  Puerto}]{BOLAND20063270}
\bibinfo{author}{Boland, N.}, \bibinfo{author}{Domínguez-Marín, P.},
  \bibinfo{author}{Nickel, S.}, \bibinfo{author}{Puerto, J.},
  \bibinfo{year}{2006}.
\newblock \bibinfo{title}{Exact procedures for solving the discrete ordered
  median problem}.
\newblock \bibinfo{journal}{Computers \& Operations Research}
  \bibinfo{volume}{33}, \bibinfo{pages}{3270--3300}.
\newblock \DOIprefix\doi{https://doi.org/10.1016/j.cor.2005.03.025}.
  \bibinfo{note}{{S}pecial Issue: Operations Research and Data Mining}.
\bibitem[{Bynum et~al.(2021)Bynum, Hackebeil, Hart, Laird, Nicholson, Siirola,
  Watson and Woodruff}]{bynum2021pyomo}
\bibinfo{author}{Bynum, M.L.}, \bibinfo{author}{Hackebeil, G.A.},
  \bibinfo{author}{Hart, W.E.}, \bibinfo{author}{Laird, C.D.},
  \bibinfo{author}{Nicholson, B.L.}, \bibinfo{author}{Siirola, J.D.},
  \bibinfo{author}{Watson, J.P.}, \bibinfo{author}{Woodruff, D.L.},
  \bibinfo{year}{2021}.
\newblock \bibinfo{title}{Pyomo--optimization modeling in python}.
  volume~\bibinfo{volume}{67}.
\newblock \bibinfo{edition}{{Third}} ed., \bibinfo{publisher}{{Springer Science
  \& Business Media}}.
\bibitem[{Chen and Hooker(2023)}]{Xinying_Chen2023}
\bibinfo{author}{Chen, X.}, \bibinfo{author}{Hooker, J.N.},
  \bibinfo{year}{2023}.
\newblock \bibinfo{title}{A guide to formulating fairness in an optimization
  model}.
\newblock \bibinfo{journal}{Annals of Operations Research}
  \bibinfo{volume}{326}, \bibinfo{pages}{581--619}.
\newblock \URLprefix \url{https://doi.org/10.1007/s10479-023-05264-y},
  \DOIprefix\doi{10.1007/s10479-023-05264-y}.
\bibitem[{Cherkesly et~al.(2025)Cherkesly, Contardo and Gruson}]{Cherkesly25}
\bibinfo{author}{Cherkesly, M.}, \bibinfo{author}{Contardo, C.},
  \bibinfo{author}{Gruson, M.}, \bibinfo{year}{2025}.
\newblock \bibinfo{title}{Ranking decomposition for the discrete ordered median
  problem}.
\newblock \bibinfo{journal}{INFORMS Journal on Computing} \bibinfo{volume}{37},
  \bibinfo{pages}{230--248}.
\bibitem[{Cho et~al.(2017)Cho, Wang, Chen, Chan and
  Swami}]{Cho2017-multiobectivesurvey}
\bibinfo{author}{Cho, J.H.}, \bibinfo{author}{Wang, Y.}, \bibinfo{author}{Chen,
  I.R.}, \bibinfo{author}{Chan, K.S.}, \bibinfo{author}{Swami, A.},
  \bibinfo{year}{2017}.
\newblock \bibinfo{title}{A survey on modeling and optimizing multi-objective
  systems}.
\newblock \bibinfo{journal}{IEEE Commun. Surv. Tutor.} \bibinfo{volume}{19},
  \bibinfo{pages}{1867--1901}.
\bibitem[{Church and ReVelle(1974)}]{Church1974-zz}
\bibinfo{author}{Church, R.}, \bibinfo{author}{ReVelle, C.},
  \bibinfo{year}{1974}.
\newblock \bibinfo{title}{The maximal covering location problem}.
\newblock \bibinfo{journal}{Papers of the Regional Science Association}
  \bibinfo{volume}{22}, \bibinfo{pages}{101--118}.
\bibitem[{Chvatal(1979)}]{chvatal1979}
\bibinfo{author}{Chvatal, V.}, \bibinfo{year}{1979}.
\newblock \bibinfo{title}{A greedy heuristic for the set-covering problem}.
\newblock \bibinfo{journal}{Mathematics of Operations Research}
  \bibinfo{volume}{4}, \bibinfo{pages}{233--235}.
\bibitem[{Cox~Jr.(2012)}]{Cox2012}
\bibinfo{author}{Cox~Jr., L.A.T.}, \bibinfo{year}{2012}.
\newblock \bibinfo{title}{Why income inequality indexes do not apply to health
  risks}.
\newblock \bibinfo{journal}{Risk Analysis} \bibinfo{volume}{32},
  \bibinfo{pages}{192--196}.
\bibitem[{Current et~al.(2002)Current, Daskin, Schilling and
  {Others}}]{Current2002-sb}
\bibinfo{author}{Current, J.}, \bibinfo{author}{Daskin, M.},
  \bibinfo{author}{Schilling, D.}, \bibinfo{author}{{Others}},
  \bibinfo{year}{2002}.
\newblock \bibinfo{title}{Discrete network location models}.
\newblock \bibinfo{journal}{Facility location: Applications and theory}
  \bibinfo{volume}{1}, \bibinfo{pages}{81--118}.
\bibitem[{Deleplanque et~al.(2020)Deleplanque, Labbé, Ponce and
  Puerto}]{Deleplanque20}
\bibinfo{author}{Deleplanque, S.}, \bibinfo{author}{Labbé, M.},
  \bibinfo{author}{Ponce, D.}, \bibinfo{author}{Puerto, J.},
  \bibinfo{year}{2020}.
\newblock \bibinfo{title}{A branch-price-and-cut procedure for the discrete
  ordered median problem}.
\newblock \bibinfo{journal}{INFORMS Journal on Computing} \bibinfo{volume}{32},
  \bibinfo{pages}{582--599}.
\bibitem[{Drezner et~al.(2009)Drezner, Drezner and Guyse}]{Drezner2009-sx}
\bibinfo{author}{Drezner, T.}, \bibinfo{author}{Drezner, Z.},
  \bibinfo{author}{Guyse, J.}, \bibinfo{year}{2009}.
\newblock \bibinfo{title}{Equitable service by a facility: Minimizing the
  {Gini} coefficient}.
\newblock \bibinfo{journal}{Comput. Oper. Res.} \bibinfo{volume}{36},
  \bibinfo{pages}{3240--3246}.
\bibitem[{Eiselt and Laporte(1995)}]{Eiselt1995-ur}
\bibinfo{author}{Eiselt, H.A.}, \bibinfo{author}{Laporte, G.},
  \bibinfo{year}{1995}.
\newblock \bibinfo{title}{Objectives in location problems}, in:
  \bibinfo{editor}{Drezner, Z.} (Ed.), \bibinfo{booktitle}{Facility Location: A
  Survey of Applications and Methods}. \bibinfo{publisher}{Springer},
  \bibinfo{address}{New York, NY}, pp. \bibinfo{pages}{151--180}.
\bibitem[{Erkut(1993)}]{Erkut1993-kw}
\bibinfo{author}{Erkut, E.}, \bibinfo{year}{1993}.
\newblock \bibinfo{title}{Inequality measures for location problems}.
\newblock \bibinfo{journal}{Comput. Oper. Res.} .
\bibitem[{Gupta et~al.(2023)Gupta, Moondra and Singh}]{Gupta23}
\bibinfo{author}{Gupta, S.}, \bibinfo{author}{Moondra, J.},
  \bibinfo{author}{Singh, M.}, \bibinfo{year}{2023}.
\newblock \bibinfo{title}{Which {Lp Norm} is the fairest? approximations for
  fair facility location across all ``p''}, in: \bibinfo{booktitle}{Proceedings
  of the 24th ACM Conference on Economics and Computation},
  \bibinfo{publisher}{Association for Computing Machinery},
  \bibinfo{address}{New York, NY, USA}. p. \bibinfo{pages}{817}.
\bibitem[{{Gurobi Optimization, LLC}(2025)}]{gurobi}
\bibinfo{author}{{Gurobi Optimization, LLC}}, \bibinfo{year}{2025}.
\newblock \bibinfo{title}{Gurobi Optimizer Reference Manual}.
\newblock \URLprefix \url{https://www.gurobi.com}.
\bibitem[{Hakimi(1964)}]{FL1}
\bibinfo{author}{Hakimi, S.L.}, \bibinfo{year}{1964}.
\newblock \bibinfo{title}{Optimum locations of switching centers and the
  absolute centers and medians of a graph}.
\newblock \bibinfo{journal}{Operations Research} \bibinfo{volume}{12},
  \bibinfo{pages}{450–459}.
\bibitem[{Hakimi(1965)}]{Hakimi1965-ay}
\bibinfo{author}{Hakimi, S.L.}, \bibinfo{year}{1965}.
\newblock \bibinfo{title}{Optimum distribution of switching centers in a
  communication network and some related graph theoretic problems}.
\newblock \bibinfo{journal}{Oper. Res.} \bibinfo{volume}{13},
  \bibinfo{pages}{462--475}.
\bibitem[{Halpern(1976)}]{Halpern1976-fz}
\bibinfo{author}{Halpern, J.}, \bibinfo{year}{1976}.
\newblock \bibinfo{title}{The location of a center-median convex combination on
  an undirected tree}.
\newblock \bibinfo{journal}{Journal of Regional Science} \bibinfo{volume}{16},
  \bibinfo{pages}{237--245}.
\bibitem[{Hart et~al.(2011)Hart, Watson and Woodruff}]{hart2011pyomo}
\bibinfo{author}{Hart, W.E.}, \bibinfo{author}{Watson, J.P.},
  \bibinfo{author}{Woodruff, D.L.}, \bibinfo{year}{2011}.
\newblock \bibinfo{title}{Pyomo: modeling and solving mathematical programs in
  python}.
\newblock \bibinfo{journal}{Mathematical Programming Computation}
  \bibinfo{volume}{3}, \bibinfo{pages}{219--260}.
\bibitem[{Horton et~al.(2024)Horton, Logan, Skipper and
  Speakman}]{horton2024hundreds}
\bibinfo{author}{Horton, D.}, \bibinfo{author}{Logan, T.},
  \bibinfo{author}{Skipper, D.}, \bibinfo{author}{Speakman, E.},
  \bibinfo{year}{2024}.
\newblock \bibinfo{title}{Hundreds of grocery outlets needed across the united
  states to achieve walkable cities}.
\newblock \href{http://arxiv.org/abs/2404.01209}{{\tt arXiv:2404.01209}}.
\bibitem[{Hsu et~al.(2021)Hsu, Sheriff, Chakraborty and Manya}]{Hsuetal2021}
\bibinfo{author}{Hsu, A.}, \bibinfo{author}{Sheriff, G.},
  \bibinfo{author}{Chakraborty, T.}, \bibinfo{author}{Manya, D.},
  \bibinfo{year}{2021}.
\newblock \bibinfo{title}{Disproportionate exposure to urban heat island
  intensity across major us cities}.
\newblock \bibinfo{journal}{Nature Communications} \bibinfo{volume}{12}.
\bibitem[{Karsu and Morton(2015)}]{Karsu2015-cb}
\bibinfo{author}{Karsu, {\"O}.}, \bibinfo{author}{Morton, A.},
  \bibinfo{year}{2015}.
\newblock \bibinfo{title}{Inequity averse optimization in operational
  research}.
\newblock \bibinfo{journal}{Eur. J. Oper. Res.} \bibinfo{volume}{245},
  \bibinfo{pages}{343--359}.
\bibitem[{Kolm(1976)}]{Kolm1976-ab}
\bibinfo{author}{Kolm, S.C.}, \bibinfo{year}{1976}.
\newblock \bibinfo{title}{Unequal inequalities. {I}}.
\newblock \bibinfo{journal}{J. Econ. Theory} \bibinfo{volume}{12},
  \bibinfo{pages}{416--442}.
\bibitem[{Lejeune and Prasad(2013)}]{Lejeune2013-tr}
\bibinfo{author}{Lejeune, M.A.}, \bibinfo{author}{Prasad, S.Y.},
  \bibinfo{year}{2013}.
\newblock \bibinfo{title}{Effectiveness-equity models for facility location
  problems on tree networks}.
\newblock \bibinfo{journal}{Networks} \bibinfo{volume}{62},
  \bibinfo{pages}{243--254}.
\bibitem[{Ljubić et~al.(2024)Ljubić, Pozo, Puerto and
  Torrejón}]{LJUBIC2024858}
\bibinfo{author}{Ljubić, I.}, \bibinfo{author}{Pozo, M.A.},
  \bibinfo{author}{Puerto, J.}, \bibinfo{author}{Torrejón, A.},
  \bibinfo{year}{2024}.
\newblock \bibinfo{title}{Benders decomposition for the discrete ordered median
  problem}.
\newblock \bibinfo{journal}{European Journal of Operational Research}
  \bibinfo{volume}{317}, \bibinfo{pages}{858--874}.
\newblock \URLprefix
  \url{https://www.sciencedirect.com/science/article/pii/S0377221724003291},
  \DOIprefix\doi{https://doi.org/10.1016/j.ejor.2024.04.030}.
\bibitem[{Logan et~al.(2021)Logan, Anderson, Williams and Conrow}]{LAWC2021}
\bibinfo{author}{Logan, T.}, \bibinfo{author}{Anderson, M.},
  \bibinfo{author}{Williams, T.}, \bibinfo{author}{Conrow, L.},
  \bibinfo{year}{2021}.
\newblock \bibinfo{title}{Measuring inequalities in urban systems: An approach
  for evaluating the distribution of amenities and burdens}.
\newblock \bibinfo{journal}{Computers, Environment and Urban Systems}
  \bibinfo{volume}{86}, \bibinfo{pages}{101590}.
\bibitem[{Logan and Guikema(2020)}]{Logan2020-refresilience}
\bibinfo{author}{Logan, T.M.}, \bibinfo{author}{Guikema, S.D.},
  \bibinfo{year}{2020}.
\newblock \bibinfo{title}{Reframing resilience: Equitable access to essential
  services}.
\newblock \bibinfo{journal}{Risk Anal.} \bibinfo{volume}{40},
  \bibinfo{pages}{1538--1553}.
\newblock \DOIprefix\doi{10.1111/risa.13492}.
\bibitem[{Luxen and Vetter(2011)}]{OSRM}
\bibinfo{author}{Luxen, D.}, \bibinfo{author}{Vetter, C.},
  \bibinfo{year}{2011}.
\newblock \bibinfo{title}{Real-time routing with openstreetmap data},
  \bibinfo{publisher}{Association for Computing Machinery}. p.
  \bibinfo{pages}{513–516}.
\bibitem[{Maguire and Sheriff(2020)}]{MS2020}
\bibinfo{author}{Maguire, K.}, \bibinfo{author}{Sheriff, G.},
  \bibinfo{year}{2020}.
\newblock \bibinfo{title}{Health risk, inequality indexes, and environmental
  justices}.
\newblock \bibinfo{journal}{Risk Analysis} \bibinfo{volume}{40},
  \bibinfo{pages}{2661–2674}.
\bibitem[{Mandell(1991)}]{Mandell1991-by}
\bibinfo{author}{Mandell, M.B.}, \bibinfo{year}{1991}.
\newblock \bibinfo{title}{Modelling effectiveness-equity trade-offs in public
  service delivery systems}.
\newblock \bibinfo{journal}{Manage. Sci.} \bibinfo{volume}{37},
  \bibinfo{pages}{467--482}.
\bibitem[{Mansur and Sheriff(2021)}]{MansurSheriff2021}
\bibinfo{author}{Mansur, E.T.}, \bibinfo{author}{Sheriff, G.},
  \bibinfo{year}{2021}.
\newblock \bibinfo{title}{On the measurement of environmental inequality:
  Ranking emissions distributions generated by different policy instruments}.
\newblock \bibinfo{journal}{Journal of the Association of Environmental and
  Resource Economists} \bibinfo{volume}{8}, \bibinfo{pages}{721--758}.
\newblock \URLprefix \url{https://doi.org/10.1086/713113},
  \DOIprefix\doi{10.1086/713113}.
\bibitem[{Marsh and Schilling(1994)}]{Marsh1994-ax}
\bibinfo{author}{Marsh, M.T.}, \bibinfo{author}{Schilling, D.A.},
  \bibinfo{year}{1994}.
\newblock \bibinfo{title}{Equity measurement in facility location analysis: A
  review and framework}.
\newblock \bibinfo{journal}{Eur. J. Oper. Res.} \bibinfo{volume}{74},
  \bibinfo{pages}{1--17}.
\bibitem[{Marín(2011)}]{MARIN201127}
\bibinfo{author}{Marín, A.}, \bibinfo{year}{2011}.
\newblock \bibinfo{title}{The discrete facility location problem with balanced
  allocation of customers}.
\newblock \bibinfo{journal}{European Journal of Operational Research}
  \bibinfo{volume}{210}, \bibinfo{pages}{27--38}.
\newblock \DOIprefix\doi{https://doi.org/10.1016/j.ejor.2010.10.012}.
\bibitem[{Marín et~al.(2009)Marín, Nickel, Puerto and Velten}]{MARIN20091128}
\bibinfo{author}{Marín, A.}, \bibinfo{author}{Nickel, S.},
  \bibinfo{author}{Puerto, J.}, \bibinfo{author}{Velten, S.},
  \bibinfo{year}{2009}.
\newblock \bibinfo{title}{A flexible model and efficient solution strategies
  for discrete location problems}.
\newblock \bibinfo{journal}{Discrete Applied Mathematics}
  \bibinfo{volume}{157}, \bibinfo{pages}{1128--1145}.
\newblock \DOIprefix\doi{https://doi.org/10.1016/j.dam.2008.03.013}.
\bibitem[{McAllister(1976)}]{mcallister1976}
\bibinfo{author}{McAllister, D.M.}, \bibinfo{year}{1976}.
\newblock \bibinfo{title}{Equity and efficiency in public facility location}.
\newblock \bibinfo{journal}{Geographical Analysis} \bibinfo{volume}{8},
  \bibinfo{pages}{47--63}.
\bibitem[{Mulligan(1991)}]{Mulligan1991-zf}
\bibinfo{author}{Mulligan, G.F.}, \bibinfo{year}{1991}.
\newblock \bibinfo{title}{Equality measures and facility location}.
\newblock \bibinfo{journal}{Pap. Reg. Sci.} \bibinfo{volume}{70},
  \bibinfo{pages}{345--365}.
\bibitem[{Neff et~al.(2009)Neff, Palmer, McKenzie and Lawrence}]{Neff09}
\bibinfo{author}{Neff, R.A.}, \bibinfo{author}{Palmer, A.M.},
  \bibinfo{author}{McKenzie, S.E.}, \bibinfo{author}{Lawrence, R.S.},
  \bibinfo{year}{2009}.
\newblock \bibinfo{title}{Food systems and public health disparities}.
\newblock \bibinfo{journal}{Journal of hunger \& environmental nutrition}
  \bibinfo{volume}{4}, \bibinfo{pages}{282--314}.
\bibitem[{Nickel(2001)}]{10.1007/978-3-642-56656-1_12}
\bibinfo{author}{Nickel, S.}, \bibinfo{year}{2001}.
\newblock \bibinfo{title}{Discrete ordered weber problems}, in:
  \bibinfo{editor}{Fleischmann, B.}, \bibinfo{editor}{Lasch, R.},
  \bibinfo{editor}{Derigs, U.}, \bibinfo{editor}{Domschke, W.},
  \bibinfo{editor}{Rieder, U.} (Eds.), \bibinfo{booktitle}{Operations Research
  Proceedings}, \bibinfo{publisher}{Springer Berlin Heidelberg},
  \bibinfo{address}{Berlin, Heidelberg}. pp. \bibinfo{pages}{71--76}.
\bibitem[{O'Brien(1969)}]{OBrien1969-wl}
\bibinfo{author}{O'Brien, R.J.}, \bibinfo{year}{1969}.
\newblock \bibinfo{title}{Model for planning the location and size of urban
  schools}.
\newblock \bibinfo{journal}{Socioecon. Plann. Sci.} \bibinfo{volume}{2},
  \bibinfo{pages}{141--153}.
\bibitem[{Ogryczak(2009)}]{Ogryczak2009-dx}
\bibinfo{author}{Ogryczak, W.}, \bibinfo{year}{2009}.
\newblock \bibinfo{title}{Inequality measures and equitable locations}.
\newblock \bibinfo{journal}{Ann. Oper. Res.} \bibinfo{volume}{167},
  \bibinfo{pages}{61--86}.
\bibitem[{Olivier et~al.(2022)Olivier, Lodi and Pesant}]{Olivier2022}
\bibinfo{author}{Olivier, P.}, \bibinfo{author}{Lodi, A.},
  \bibinfo{author}{Pesant, G.}, \bibinfo{year}{2022}.
\newblock \bibinfo{title}{Measures of balance in combinatorial optimization}.
\newblock \bibinfo{journal}{4OR} \bibinfo{volume}{20},
  \bibinfo{pages}{391--415}.
\bibitem[{{OpenStreetMap contributors}(2017)}]{OpenStreetMap}
\bibinfo{author}{{OpenStreetMap contributors}}, \bibinfo{year}{2017}.
\newblock \bibinfo{title}{{Planet dump retrieved from https://planet.osm.org
  }}.
\newblock \bibinfo{howpublished}{\url{ https://www.openstreetmap.org }}.
\bibitem[{Robinson et~al.(2022)Robinson, Redvers, Camargo, Bosch, Breed,
  Brenner, Carney, Chauhan, Dasari, Dietz, Friedman, Grieneisen, Hoisington,
  Horve, Hunter, Jech, Jorgensen, Lowry, Man, Mhuireach, Navarro-Pérez,
  Ritchie, Stewart, Watkins, Weinstein and Ishaq}]{MicrobialExposure}
\bibinfo{author}{Robinson, J.M.}, \bibinfo{author}{Redvers, N.},
  \bibinfo{author}{Camargo, A.}, \bibinfo{author}{Bosch, C.A.},
  \bibinfo{author}{Breed, M.F.}, \bibinfo{author}{Brenner, L.A.},
  \bibinfo{author}{Carney, M.A.}, \bibinfo{author}{Chauhan, A.},
  \bibinfo{author}{Dasari, M.}, \bibinfo{author}{Dietz, L.G.},
  \bibinfo{author}{Friedman, M.}, \bibinfo{author}{Grieneisen, L.},
  \bibinfo{author}{Hoisington, A.J.}, \bibinfo{author}{Horve, P.F.},
  \bibinfo{author}{Hunter, A.}, \bibinfo{author}{Jech, S.},
  \bibinfo{author}{Jorgensen, A.}, \bibinfo{author}{Lowry, C.A.},
  \bibinfo{author}{Man, I.}, \bibinfo{author}{Mhuireach, G.},
  \bibinfo{author}{Navarro-Pérez, E.}, \bibinfo{author}{Ritchie, E.G.},
  \bibinfo{author}{Stewart, J.D.}, \bibinfo{author}{Watkins, H.},
  \bibinfo{author}{Weinstein, P.}, \bibinfo{author}{Ishaq, S.L.},
  \bibinfo{year}{2022}.
\newblock \bibinfo{title}{Twenty important research questions in microbial
  exposure and social equity}.
\newblock \bibinfo{journal}{mSystems} \bibinfo{volume}{7},
  \bibinfo{pages}{e01240--21}.
\bibitem[{Rocco et~al.(2023)Rocco, Nock and Barker}]{10014676}
\bibinfo{author}{Rocco, C.M.}, \bibinfo{author}{Nock, D.},
  \bibinfo{author}{Barker, K.}, \bibinfo{year}{2023}.
\newblock \bibinfo{title}{A fairness-based approach to network restoration}.
\newblock \bibinfo{journal}{IEEE Transactions on Systems, Man, and Cybernetics:
  Systems} \bibinfo{volume}{53}, \bibinfo{pages}{3890--3894}.
\bibitem[{Savas(1978)}]{Savas1978-pk}
\bibinfo{author}{Savas, E.S.}, \bibinfo{year}{1978}.
\newblock \bibinfo{title}{On equity in providing public services}.
\newblock \bibinfo{journal}{Manage. Sci.} \bibinfo{volume}{24},
  \bibinfo{pages}{800--808}.
\bibitem[{Skipper et~al.(2025)Skipper, Agarwala, Murrell, Rosenberg, Wilson and
  Logan}]{Skipper2025}
\bibinfo{author}{Skipper, D.}, \bibinfo{author}{Agarwala, S.},
  \bibinfo{author}{Murrell, J.}, \bibinfo{author}{Rosenberg, C.},
  \bibinfo{author}{Wilson, D.}, \bibinfo{author}{Logan, T.M.},
  \bibinfo{year}{2025}.
\newblock \bibinfo{title}{Optimizing fair geographic access to polling}.
\newblock \bibinfo{journal}{Election Law Journal: Rules, Politics, and Policy}
  \DOIprefix\doi{https://doi.org/10.1089/elj.2024.0060}.
\bibitem[{Svirsko et~al.(2022)Svirsko, Logan, Domanowski and
  Skipper}]{Svirsko2022-ns}
\bibinfo{author}{Svirsko, A.C.}, \bibinfo{author}{Logan, T.M.},
  \bibinfo{author}{Domanowski, C.}, \bibinfo{author}{Skipper, D.},
  \bibinfo{year}{2022}.
\newblock \bibinfo{title}{Developing robust facility reopening processes
  following natural disasters}.
\newblock \bibinfo{journal}{Oper. Res. Forum} \bibinfo{volume}{3}.
\newblock \DOIprefix\doi{10.1007/s43069-022-00147-7}.

\end{thebibliography}

\appendix

\section{Theorem and Proposition Proofs}
\label{sec:app}


\newtheorem*{prop:equivalent}{Proposition \ref{prop:equivalent}}

\begin{prop:equivalent} Suppose $y_{r,s} \in \{0,1\}$ for all $s \in S$, $r\in R$, and  $\sum_{s\in S} y_{r,s} 
= 1$ for all $r \in R$. Then 
\begin{equation*}
    \sum_{r \in R}p_re^{-\kappa \sum_{s \in S}y_{r,s}d_{r,s}}\\
    ~=~\sum_{r \in R}\sum_{s \in S}p_ry_{r,s}e^{-\kappa d_{r,s}}.
\end{equation*}
\end{prop:equivalent}

\begin{proof}
Fix $r \in R$. There is exactly one $s\in S$ such that $y_{r,s}=1$ because of our assumption that $\sum_{s\in S} y_{r,s} = 1$. Let $s'\in S$ be that index so that we have $y_{r,s'}=1$ and $y_{r,s}=0$ for all $s\neq s'$.
\begingroup
\allowdisplaybreaks
Then,
\begin{align*}
    p_re^{-\kappa \sum_{s \in S}y_{r,s}d_{r,s}}&=p_re^{-\kappa(0 d_{r,1}+\dots+ 1 d_{r,s'}+\dots+0 d_{r,\abs{S}})}\\
    &=p_re^{-\kappa d_{r,s'}}\\
&=0 p_r e^{-\kappa d_{r,1}}+\dots+ 1 p_r e^{-\kappa d_{r,s'}}+\dots+0 p_r e^{-\kappa d_{r,\abs{S}}}\\
    &=\sum_{s \in S}p_r y_{r,s} e^{-\kappa d_{r,s}}.
\end{align*}
Since $r \in R$ was chosen arbitrarily, it follows that the desired equality holds. 
\endgroup
\end{proof}


\newtheorem*{thm:penalty}{Theorem \ref{thm:penalty}}

\begin{thm:penalty} 
Let $\overline{\mathcal{K}}$ represent the unpenalized objective value of (KPL). Let $\mathcal{K}$ represent the associated unpenalized Kolm-Pollak score: $\mathcal{K} = -\frac{1}{\kappa}\ln\left(\frac{1}{T}\overline{\mathcal{K}}\right)$. Then adding a penalty of,
\begin{equation*}
\rho ~:=~ Te^{-\kappa \mathcal{K}}(e^{-\kappa \sigma} -1),
\end{equation*}
to $\overline{\mathcal{K}}$ is equivalent to adding a penalty of $\sigma \geq 0$ units to $\mathcal{K}$.
\end{thm:penalty}
\begin{proof}
Note that
$\overline{\mathcal{K}} = Te^{-\kappa \mathcal{K}}$.
Converting the penalized linear objective function value to a Kolm-Pollak score, we obtain the appropriately penalized Kolm-Pollak value:

\hspace*{18mm} $\displaystyle
-\frac{1}{\kappa}\ln\left(\frac{1}{T}\left(\overline{\mathcal{K}} + \rho\right)\right)~ 
~=~ -\frac{1}{\kappa}\ln\left(e^{-\kappa \mathcal{K}} + e^{-\kappa \mathcal{K}}(e^{-\kappa \sigma} -1)\right)
~=~ \mathcal{K} + \sigma.$ \end{proof} 


\newtheorem*{thm:penalty_approx_error}{Theorem \ref{thm:penalty_approx_error}}

\begin{thm:penalty_approx_error} 
The error in the penalty applied to an optimal Kolm-Pollak score by (KPL$^p$) is,
\begin{equation*}
\hat{E}(\Delta,\sigma^*) 
~:=~ \hat{\sigma} - \sigma^* 
~=~ -\frac{1}{\kappa}\ln\left(e^{\kappa\sigma^*}+e^{-\kappa\Delta}(1-e^{\kappa\sigma^*})\right).
\end{equation*}
\end{thm:penalty_approx_error}

\begin{proof} Suppose $(\mathbf{x}^*, \mathbf{y}^*)$ is optimal to (KPL$^p$). Let $\overline{\mathcal{K}}^* := \overline{\mathcal{K}}(\mathbf{y}^*)$. Then,
\begin{align*}
\mathcal{K}^* + \hat{\sigma}
~&=~ -\frac{1}{\kappa} \ln\left(\frac{1}{T}\left(\overline{\mathcal{K}}^* + Te^{-\kappa \hat{\mathcal{K}}}(e^{-\kappa \sigma^*} - 1)\right)\right) \\
&=~ -\frac{1}{\kappa} \ln\left(\frac{1}{T}\left(Te^{-\kappa \mathcal{K}^*} + Te^{-\kappa \hat{\mathcal{K}}}(e^{-\kappa \sigma^*} - 1)\right)\right) \\
&=~ -\frac{1}{\kappa} \ln\left(e^{-\kappa \mathcal{K}^*} + e^{-\kappa \hat{\mathcal{K}}-\kappa \sigma^*} - e^{-\kappa \hat{\mathcal{K}}}\right) \\
&=~ -\frac{1}{\kappa} \ln\left(e^{-\kappa \hat{\mathcal{K}}-\kappa \sigma^*}(1 + e^{-\kappa \mathcal{K}^* + \kappa \hat{\mathcal{K}} + \kappa \sigma^*} - e^{\kappa \sigma^*})\right) \\
&=~ \hat{\mathcal{K}} + \sigma^* - \frac{1}{\kappa}\ln\left(1-e^{\kappa \sigma^*} + e^{\kappa \sigma^*}e^{\kappa(\hat{\mathcal{K}}-\mathcal{K}^*)}\right),
\end{align*}
so that,
\begin{align*}
\hat{\sigma} - \sigma^*
~&=~ \hat{\mathcal{K}} - \mathcal{K}^* - \frac{1}{\kappa}\ln\left(1-e^{\kappa \sigma^*} + e^{\kappa \sigma^*}e^{\kappa(\hat{\mathcal{K}}-\mathcal{K}^*)}\right) \\
&=~ \Delta - \frac{1}{\kappa}\ln\left(1-e^{\kappa \sigma^*} + e^{\kappa \sigma^*}e^{\kappa\Delta}\right) \\
&=~ \Delta - \frac{1}{\kappa}\ln\left(e^{\kappa\Delta}(e^{-\kappa\Delta}(1-e^{\kappa\sigma^*}) + e^{\kappa\sigma^*})\right) \\
&=~ -\frac{1}{\kappa}\ln\left(e^{\kappa\sigma^*}+e^{-\kappa\Delta}(1-e^{\kappa\sigma^*})\right).
\end{align*}
\end{proof}


\newtheorem*{thm:penalty_approx_error_bound}{Theorem \ref{thm:penalty_approx_error_bound}}

\begin{thm:penalty_approx_error_bound} 
Suppose $\sigma^* > 0$ is fixed:
\begin{enumerate}[label=(\arabic*)]
\item $\hat{E}(0,\sigma^*) = 0$ and $|\hat{E}(\Delta,\sigma^*)|$ is increasing in $|\Delta|$.
\item If $\hat{\mathcal{K}} > \mathcal{K}^*$ then $\hat{\sigma} > \sigma^*$ (locations are over-penalized) and 
\[
0
< |\Delta|(1-e^{\kappa \sigma^*}) 
< |\hat{E}(\Delta,\sigma^*)|
< |\Delta|.
\]
\item If $\hat{\mathcal{K}} < \mathcal{K}^*$ then $\hat{\sigma} < \sigma^*$ (locations are under-penalized) and 
\[
0
< |\hat{E}(\Delta,\sigma^*)|
< \text{min}\left\{|\Delta|(1-e^{\kappa \sigma^*}),\sigma^*\right\}.
\]
\end{enumerate}
Suppose $\Delta \in \mathbb{R}$ is fixed:
\begin{enumerate}[resume*]
\item $\hat{E}(\Delta, 0) = 0$ and $|\hat{E}(\Delta, \sigma^*)|$ is increasing in $\sigma^*$.
\end{enumerate}
\end{thm:penalty_approx_error_bound}

\begin{proof}
(1) Suppose $\sigma^* > 0$ is fixed. We have 
\[
\frac{\delta}{\delta \Delta} \hat{E}(\Delta,\sigma^*) 
~=~ \frac{e^{-\kappa\Delta}(1-e^{\kappa\sigma^*})}{e^{-\kappa\Delta}(1-e^{\kappa\sigma^*}) + e^{\kappa\sigma^*}}
~=~ \frac{e^{-\kappa\Delta}}{e^{-\kappa\Delta} + \frac{1}{e^{-\kappa\sigma^*}-1}}.
\]
Recalling that $\kappa < 0$, we have that $e^{-\kappa\sigma^*} > 1$, so that $\frac{\delta}{\delta \Delta} \hat{E}(\Delta,\sigma^*) \in (0,1)$. Thus, $\hat{E}(\Delta,\sigma^*)$ is increasing (though not quickly) in $\Delta$. It is straightforward to evaluate $\hat{E}(0,\sigma^*)=0$. Thus, when $\Delta>0$, $\hat{E}(\Delta,\sigma^*)>0$ and $\hat{E}(\Delta,\sigma^*)=|\hat{E}(\Delta,\sigma^*)|$ is increasing in $|\Delta|$. Similarly, when $\Delta<0$, $\hat{E}(0,\sigma^*)<0$ and $-\hat{E}(\Delta,\sigma^*) = |\hat{E}(\Delta,\sigma^*)|$ is increasing in $|\Delta|$.

(2) Suppose $\sigma^*>0$ is fixed. From above, $\hat{E}(\Delta,\sigma^*) > 0$ when $\Delta>0$. The tangent line to the univariate function $\hat{E}(\Delta,\sigma^*)$ at $\Delta = 0$ is $\ell(\Delta)=\Delta(1 - e^{\kappa\sigma^*})$. Recalling again that $\kappa < 0$,
\[
\frac{\delta^2}{\delta \Delta^2} \hat{E}(\Delta,\sigma^*) 
~=~ \frac{-\kappa e^{-\kappa\Delta}}{(e^{-\kappa \sigma^*}-1)\left(e^{-\kappa \Delta}-\frac{1}{e^{-\kappa \sigma^*}-1}\right)^2} 
~>~ 0,
\]
for all $\Delta \in \mathbb{R}$. Thus, $\Delta(1-e^{\kappa \sigma^*}) \leq \hat{E}(\Delta,\sigma^*)$ for all $\Delta \in \mathbb{R}$, and, in particular, when $\Delta>0$, $|\Delta|(1-e^{\kappa\sigma^*})=\Delta(1-e^{\kappa\sigma^*})<\hat{E}(\Delta,\sigma^*)=|\hat{E}(\Delta,\sigma^*)|$. 
From the proof of Theorem \ref{thm:penalty_approx_error}, 
\[
\hat{E}(\Delta,\sigma^*)
~=~ \Delta - \frac{1}{\kappa}\ln\left(1-e^{\kappa \sigma^*} + e^{\kappa \sigma^*}e^{\kappa\Delta}\right) \\
~=~ \Delta - \frac{1}{\kappa}\ln\left(1-e^{\kappa \sigma^*}(1 - e^{\kappa\Delta})\right).
\]
In this case, $\kappa\sigma^* < 0$ and $\kappa \Delta < 0$, so $e^{\kappa \sigma^*}(1 - e^{\kappa\Delta}) \in (0,1)$. Thus, $|\hat{E}(\Delta,\sigma^*)| = \hat{E}(\Delta,\sigma^*) < \Delta = |\Delta|$.

(3) For $\Delta < 0$, we have that $\hat{E}(\Delta, \sigma^*) < 0$, so $\Delta(1-e^{\kappa \sigma^*}) \leq \hat{E}(\Delta,\sigma^*)$ implies that $|\hat{E}(\Delta, \sigma^*)| < |\Delta|(1 - e^{\kappa\sigma^*})$. Because $\hat{E}(\Delta,\sigma^*)$ is increasing in $\Delta$ and $\lim_{\Delta \rightarrow -\infty} \hat{E}(\Delta, \sigma^*) = -\sigma^*$, $|\hat{E}(\Delta,\sigma^*)| < \sigma^*$.

(4) Now suppose that $\Delta \in \mathbb{R}$ is fixed. We can evaluate $\hat{E}(\Delta,0)=0$. To show that $|\hat{E}(\Delta,\sigma^*)|$ is increasing in $\sigma^*$, evaluate 
\[
\frac{\delta}{\delta\sigma^*} \hat{E}(\Delta,\sigma^*) 
~=~ \frac{e^{-\kappa\Delta}e^{\kappa\sigma^*} - e^{\kappa\sigma^*}}{e^{-\kappa\Delta} + e^{\kappa\sigma^*} - e^{-\kappa\Delta}e^{\kappa\sigma^*}} 
~=~ \frac{e^{\kappa\sigma^*}}{\frac{1}{1-e^{\kappa\Delta}} - e^{\kappa\sigma^*}}.
\]
Recalling that $\kappa < 0$, we have $e^{\kappa\sigma^*} \in (0,1)$. Suppose $\Delta > 0$. Then $\kappa\Delta < 0$, which means $\frac{1}{1-e^{\kappa\Delta}} > 1$. It follows that $\frac{\delta}{\delta\sigma^*} \hat{E}(\Delta,\sigma^*) > 0$. From above, we know that $\hat{E}(\Delta, \sigma^*)>0$ in this case, so $|\hat{E}(\Delta, \sigma^*)| ~=~ \hat{E}(\Delta,\sigma^*)$ is increasing in $\sigma^*$. Suppose $\Delta < 0$. In this case, $\kappa\Delta > 0$, so $\frac{1}{1-e^{\kappa\Delta}} < 0$, which means $\frac{\delta}{\delta\sigma^*} \hat{E}(\Delta,\sigma^*) < 0$ and $\hat{E}(\Delta,\sigma^*)$ is decreasing in $\sigma^*$. However, when $\Delta < 0$, $\hat{E}(\Delta,\sigma^*) < 0$, so $|\hat{E}(\Delta, \sigma^*)|$ is increasing in $\sigma^*$.

\end{proof}


\newtheorem*{prop:Khats}{Proposition \ref{prop:Khats}}

\begin{prop:Khats} 
Assuming the same value of the parameter $\alpha$ is used in (KPL)$^{all}$, (KPL)$^{rem}$, and (KPL$^p$), 
\[ \mathcal{K}^{all} ~\leq~ \mathcal{K}^* ~\leq~ \mathcal{K}^{rem}.
\]
\end{prop:Khats}
\begin{proof}
The optimal solution to (KPL$^p$) is feasible to (KPL)$^{all}$. Thus, $\mathcal{K}^{all} \leq \mathcal{K}^*$.
The optimal solution to (KPL)$^{rem}$ is feasible to (KPL$^p$) and corresponds to the same objective function value in both (KPL)$^{rem}$ and (KPL$^p$) because none of the locations are penalized in either model. Thus, $\mathcal{K}^* \leq \mathcal{K}^{rem}$.
\end{proof}


\newtheorem*{thm:more_bounds}{Theorem \ref{thm:more_bounds}}

\begin{thm:more_bounds} 
Suppose $\hat{\mathcal{K}} = \mathcal{K}^{all}$ in (KPL$^p$). Then undesirable locations are under-penalized, and
\[
\sigma^* - \hat{\sigma}
~\leq~ |\Delta|(1-e^{\kappa \sigma^*}) 
~\leq~ \sigma^{all}(1-e^{\kappa \sigma^*}) 
~\leq~ \sigma^{all}(1-e^{\kappa \sigma^{all}}).
\]
\end{thm:more_bounds}

\begin{proof}
In this case, $\Delta \leq 0$ and the first inequality follows from Theorem \ref{thm:penalty_approx_error_bound} and Proposition \ref{prop:Khats}. The second and third inequalities require demonstrating that $|\Delta| \leq \sigma^{all}$ and $\sigma^* \leq \sigma^{all}$, respectively.

$(\mathbf{x}^{all}, \mathbf{y}^{all})$ is feasible to (KPL$^p$). If $(\mathbf{x}^{all}, \mathbf{y}^{all})$ is optimal to (KPL$^p$), then $\Delta = 0$ and there is no error in the penalty approximation: $\hat{E}(0,\sigma^*) = 0$. In other words, $Te^{-\kappa(\mathcal{K}^{all}+\sigma^{all})}$ represents an upper bound on the optimal objective value of (KPL$^p$):
\[
Te^{-\kappa(\mathcal{K}^* +\hat{\sigma})} ~\leq~ Te^{-\kappa(\mathcal{K}^{all}+\sigma^{all})}.
\]
Thus,
\[
\mathcal{K}^* + \sigma^* + \hat{E}(\Delta, \sigma^*) ~=~ \mathcal{K}^* + \hat{\sigma} ~\leq~ \mathcal{K}^{all} + \sigma^{all}.
\]
Rearranging, we have $|\Delta| = \mathcal{K}^* - \mathcal{K}^{all} \leq \sigma^{all} - \hat{\sigma}$. By Theorem \ref{thm:penalty_approx_error_bound}, $\hat{\sigma} \geq 0$, so we have $|\Delta| \leq \sigma^{all}$.

Next, using the facts that $\Delta \leq 0$, $\kappa < 0$, and $\sigma^* \geq 0$, we have,
\[
\hat{E}(\Delta,\sigma^*) ~=~ \Delta - \frac{1}{\kappa}\ln(1-e^{\kappa\sigma^*}(1-e^{\kappa\Delta})) ~\geq~ \Delta.
\]
Finally,
\begin{align*}
\mathcal{K}^{all} + \sigma^*
&~=~ \mathcal{K}^* + \sigma^* + \mathcal{K}^{all} - \mathcal{K}^* \\
&~=~ \mathcal{K}^* + \sigma^* + \Delta \\
&~\leq~ \mathcal{K}^* + \sigma^* + \hat{E}(\Delta,\sigma^*) \\
&~\leq~ \mathcal{K}^{all} + \sigma^{all},
\end{align*}
which means that $\sigma^* \leq \sigma^{all}$, as required.
\end{proof}


\newtheorem*{prop:no_exp_error}{Proposition \ref{prop:no_exp_error}}

\begin{prop:no_exp_error}
If $c_s = c$ for all $s \in U$, and $\beta_i := -\kappa c i$, for $i = 0, 1, \dots, \min\{k,|U|\}$, then $\ddot{\sigma} = \hat{\sigma}$.
\end{prop:no_exp_error}
\begin{proof}
Due to constraints \eqref{xbinary}, for all feasible $\mathbf{x} :=(x_1, x_2, \dots, x_{|S|})$, $q = -\kappa\sum_{s \in U} c x_s$ can only take the values of the selected linearization points. At these values of $q$, the approximating tangent lines match the exponential expression, $e^q$, exactly.
\end{proof}


\newtheorem*{thm:max_ex_error}{Proposition \ref{thm:max_ex_error}}

\begin{thm:max_ex_error} 
For $b\in \mathbb{R}$, let $L_b := e^b + e^b(x-b)$ represent the tangent line to $f(x) = e^x$ at the point $(b,e^b)$. Let $g: [a,a+w] \rightarrow \mathbb{R}$ be the piecewise linear function $g(x) := \max\{L_a(x), L_{a+w}(x)\}$. The maximum error between $e^x$ and $g(x)$ on $[a,a+w]$ is, 
\[
\max_{x \in [a,a+w]} e^x - g(x) = e^a\left(e^{\frac{we^w}{e^w-1}-1} - \frac{we^w}{e^w-1}\right).
\]
\end{thm:max_ex_error}
\begin{proof}
The maximum error occurs at the intersection of lines $L_a$ and $L_{a+w}$; i.e., at $x =  a+\frac{we^w}{e^w-1}-1$. Some algebra leads to the expression above.
\end{proof}


\newtheorem*{thm:max_practical_exponent}{Proposition \ref{thm:max_practical_exponent}}

\begin{thm:max_practical_exponent} 
Let $D = \{d_r: r \in R\}$ be the distribution of distances that is used to calculate $\alpha$, and let $\mu_D$ represent the population weighted mean of $D$, $\mu_D := \frac{1}{T}\sum_{r \in R} p_r d_r$.
If $\sigma_{max} \leq \mu_D$, then $q = -\kappa \sum_{s\in U} c_s x_s \leq |\epsilon|$.
\end{thm:max_practical_exponent}
\begin{proof}
\begin{align*}
-\kappa \sum_{s \in U} c_s x_s ~\leq~ -\kappa \sigma^{max} ~\leq~ -\kappa \mu_D 
~=~ -\epsilon \alpha \mu_D 
~=~ -\epsilon \frac{(\sum_{r \in R} p_r d_r)^2}{T\sum_{r \in R} p_r d_r^2} 
~\leq~ |\epsilon|,
\end{align*}
where the final inequality is due to the Cauchy-Schwarz inequality.
\end{proof}


\newtheorem*{thm:penalty_error_tangent}{Theorem \ref{thm:penalty_error_tangent}}

\begin{thm:penalty_error_tangent} 
For  $w \in (0,1)$, suppose $\beta_i = iw$ for $i = 0, 1, \dots, n$, where  $nw > -\kappa \sigma^{max}$. Then $\ddot{\sigma} \leq \hat{\sigma}$ and
\begin{align}
|\ddot{E}(\boldsymbol{\beta},\sigma^*)| 
~:=~ \hat{\sigma} - \ddot{\sigma}
&~\leq~ 
\frac{1}{\kappa}\ln\left(1 - \frac{A(w)}{1 + e^{\kappa\sigma^*}(e^{\kappa \Delta} - 1)}\right) \tag{\ref{eq:exact_tangent_bound}} \\
&~\leq~
\begin{cases}
\frac{1}{\kappa}\ln\left(1-A(w)\right), & \text{ if } \Delta \leq 0;  \\
\frac{1}{\kappa}\ln\left(1-1.25A(w)\right),
& \text{ if } \Delta > 0 \text{ and } \hat{\mathcal{K}} = \mathcal{K}^{rem},
\end{cases} \tag{\ref{eq:tangent_bound_cases}}
\end{align}
where $A(w) := e^{\frac{we^w}{e^w-1}-1} - \frac{we^w}{e^w-1}$.
\end{thm:penalty_error_tangent}

\begin{proof}
We begin by deriving an expression for $\hat{\sigma}$ and a bound for $\ddot{\sigma}$. First,
\begin{align*}
\mathcal{K}^* + \hat{\sigma} 
&~=~ -\frac{1}{\kappa}\ln\left[\frac{1}{T}\left(\bar{\mathcal{K}}^* + Te^{-\kappa\hat{\mathcal{K}}}(e^{-\kappa\sigma^*} -1)\right)\right] \\
&~=~ -\frac{1}{\kappa}\ln\left(e^{-\kappa\mathcal{K}^*} + e^{-\kappa\hat{\mathcal{K}}}(e^{-\kappa\sigma^*} -1)\right), \\
&~=~ \mathcal{K}^* - \frac{1}{\kappa}\ln\left(1 + e^{-\kappa \Delta}(e^{-\kappa \sigma^*} - 1)\right),
\end{align*}
so that $\hat{\sigma} = - \frac{1}{\kappa}\ln\left(e^{-\kappa(\Delta+\sigma^*)} + 1 - e^{-\kappa \Delta}\right)$. Let $err(q^*, w)$ represent the error between $e^{q^*}$, $q^*=-\kappa \sigma^*$, and the lower bound provided by constraints \eqref{cons:tangent_lines}. We have,
\begin{align*}
\mathcal{K}^* + \ddot{\sigma} 
&~=~ -\frac{1}{\kappa}\ln\left[\frac{1}{T}\left(\bar{\mathcal{K}}^* + Te^{-\kappa\hat{\mathcal{K}}}(e^{-\kappa\sigma^*} - err(q^*,w) -1)\right)\right] \\
&~=~ \mathcal{K}^* - \frac{1}{\kappa}\ln\left(1 + e^{-\kappa \Delta}(e^{-\kappa \sigma^*} - err(q^*,w) - 1)\right).
\end{align*}
According to Proposition \ref{thm:max_ex_error}, $0 \leq err(q^*,w) \leq e^{\beta_i}A(w) \leq e^{-\kappa\sigma^*}A(w)$, where $q^* \in [\beta_i, \beta_{i+1})$. Note that $A(w)$ is increasing for $w \geq 0$, $A(0) = 0$, and $A(1) \approx 0.208$. Thus,
\begin{align*}
\ddot{\sigma} 
&~=~ -\frac{1}{\kappa}\ln\left(e^{-\kappa(\Delta+\sigma^*)} + 1 - e^{-\kappa\Delta} - e^{-\kappa\Delta}err(q^*,w)\right) \\
&~\geq~ -\frac{1}{\kappa}\ln\left(e^{-\kappa(\Delta+\sigma^*)} + 1 - e^{-\kappa\Delta} - e^{-\kappa(\Delta+\sigma^*)}A(w)\right)  \\
&~\geq~ 0.
\end{align*}
We have $\ddot{\sigma} \leq \hat{\sigma}$ because $err(q^*,w) \geq 0$, and for Equation \eqref{eq:exact_tangent_bound},
\begin{align*}
\hat{\sigma} - \ddot{\sigma}
&~\leq~ - \frac{1}{\kappa}\ln\left(e^{-\kappa(\Delta+\sigma^*)} + 1 - e^{-\kappa \sigma^*}\right) + \frac{1}{\kappa}\ln\left(e^{-\kappa(\Delta+\sigma^*)} + 1 - e^{-\kappa\Delta} - e^{-\kappa(\Delta+\sigma^*)}A(w)\right) \\
&~=~ \frac{1}{\kappa}\ln\left(\frac{e^{-\kappa(\Delta+\sigma^*)} + 1 - e^{-\kappa\Delta} - e^{-\kappa(\Delta+\sigma^*)}A(w)}{e^{-\kappa(\Delta+\sigma^*)} + 1 - e^{-\kappa \sigma^*}}\right) \\
&~=~ \frac{1}{\kappa}\ln\left(1 - \frac{e^{-\kappa(\Delta+\sigma^*)}A(w)}{e^{-\kappa(\Delta+\sigma^*)} + 1 - e^{-\kappa \sigma^*}}\right) \\
&~=~ \frac{1}{\kappa}\ln\left(1 - \frac{A(w)}{1 + e^{\kappa(\Delta+\sigma^*)} - e^{\kappa \sigma^*}}\right) \\
&~=~ \frac{1}{\kappa}\ln\left(1 - \frac{A(w)}{1 + e^{\kappa\sigma^*}(e^{\kappa \Delta} - 1)}\right).
\end{align*}
For the first part of \eqref{eq:tangent_bound_cases}, if $\Delta \leq 0$, then $1 + e^{\kappa\sigma^*}(e^{\kappa \Delta} - 1) \geq 1$, so that $\hat{\sigma} - \ddot{\sigma} \leq \frac{1}{\kappa}\ln\left(1 - A(w)\right)$. For the second part of \eqref{eq:tangent_bound_cases}, if $\hat{\mathcal{K}} = \mathcal{K}^{rem}$, then $\mathcal{K}^* + \sigma^* \leq \mathcal{K}^* + \hat{\sigma} \leq \mathcal{K}^{rem} + 0$, where the first inequality is due to Theorem \ref{thm:penalty_approx_error_bound} and the second is due to the optimality of (KPL$^p$). Thus, $0 \leq \sigma^* \leq \mathcal{K}^{rem} - \mathcal{K}^* = \Delta$, and we have,
\[
1 + e^{\kappa\sigma^*}(e^{\kappa \Delta} - 1)
~\geq~
1 + e^{\kappa\Delta}(e^{\kappa \Delta} - 1)
~\geq~
0.75.
\]
\end{proof}

\end{document}